\documentclass[11pt]{article}

\usepackage{amsthm,amsfonts,amssymb,amsmath,color}

\numberwithin{equation}{section}

\renewcommand\d{\partial}
\renewcommand\a{\alpha}
\renewcommand\b{\beta}
\renewcommand\o{\omega}

\newcommand\R{\mathbb R}\newcommand\N{\mathbb N}\newcommand\Z{\mathbb Z}

\def\de{\delta}

\def\O{\Omega}
\def\th{\theta}

\def\l{\lambda}
\def\vp{\varphi}
\def\epsilon{\varepsilon}
\def\e{\varepsilon}

\def\L{\Lambda}

\newcommand\br{\begin{rem}}
\newcommand\er{\end{rem}}
\newcommand\bp{\begin{pmatrix}}
\newcommand\ep{\end{pmatrix}}
\newcommand\be{\begin{equation}}
\newcommand\ee{\end{equation}}
\newcommand\ba{\begin{equation}\begin{aligned}}
\newcommand\ea{\end{aligned}\end{equation}}

\newcommand\nn{\nonumber}

\usepackage{layout}

\newcommand{\CalB}{\mathcal{B}}

\newcommand{\supp}{{\rm supp }}

\newcommand{\wto}{{\widetilde\Omega}}

\newcommand{\uu}{{\mathbf u}}
\newcommand{\vv}{{\mathbf v}}

\newcommand{\bg}{{\mathbf G}}
\newcommand{\ww}{{\mathbf w}}
\newcommand{\hh}{{\mathbf h}}
\newcommand{\bgg}{{\mathbf g}}

\newcommand{\vu}{\vc{u}}

\newcommand{\vc}[1]{{\bf #1}}

\newcommand{\dx}{{\rm d} {x}}

\newcommand{\dive}{{\rm div\,}}

\newtheorem{defi}{Definition}[section]
\newtheorem{theorem}[defi]{Theorem}
\newtheorem{proposition}[defi]{Proposition}
\newtheorem{lemma}[defi]{Lemma}
\newtheorem{corollary}[defi]{Corollary}
\newtheorem{remark}[defi]{Remark}

\numberwithin{equation}{section}

\begin{document}

\title{Uniform estimates for Stokes equations in a domain with a small hole and applications in homogenization problems}

\author{Yong Lu \footnote{Department of Mathematics, Nanjing University, 22 Hankou Road, Gulou District, 210093 Nanjing, China. Email: luyong@nju.edu.cn.}
\thanks{The author warmly thanks Eduard Feireisl and Christophe Prange for interesting discussions. The author particular thanks Jiaqi Yang for providing the idea to deal with the case $p=d, \ p=d'$ when $d\geq 3$. The work of the author was partially supported by project ANR JCJC BORDS
funded by l'ANR of France. The author acknowledges the partial support of the project LL1202 in the program ERC-CZ funded by the Ministry of Education, Youth and Sports of the Czech Republic.}}

\date{}

\maketitle

\begin{abstract}

We consider the Dirichlet problem of the Stokes equations in a domain with a shrinking hole in $\R^d, \ d\geq 2$.  A typical observation is that, the Lipschitz norm of the domain goes to infinity as the size of the hole goes to zero.  Thus, if $p\neq 2$, the classical results indicate that the  $W^{1,p}$ estimate of the solution may go to infinity as the size of the hole tends to zero. In this paper, we give a complete description for the uniform $W^{1,p}$ estimates of the solution for all $1<p<\infty$. We show that the uniform $W^{1,p}$ estimate holds if and only if $d'<p<d$ ($p=2$ when $d=2$).  We then give two applications in the study of homogenization problems in fluid mechanics: a generalization of the restriction operator and a construction of Bogovskii type operator in perforated domains with a quantitative estimate of the operator norm. 


\end{abstract}


 \tableofcontents

\renewcommand{\refname}{References}


\section{Introduction}

We consider the following Dirichlet problem of the Stokes equations with a divergence form source term in a bounded domain with a small hole:
\be\label{2}
\left\{\begin{aligned}
-\Delta \vv+\nabla \pi &=\dive \bg,\quad &&\mbox{in}~\O_\e:=\O \setminus (\e  \overline T),\\
\dive \vv&=0 ,\quad &&\mbox{in}~\O_\e=\O \setminus (\e  \overline T),\\
\vv&=0,\quad &&\mbox{on} ~\d\O_\e = \d \O\cup \e (\d T).
\end{aligned}\right.
\ee

This study is mainly motivated by the study of homogenization problems in the framework of fluid mechanics where the main goal is to describe the asymptotic behavior of fluid flows in domains perforated with a large number of tiny holes. As observed by Allaire \cite{ALL-NS1,ALL-NS2}, it is important to consider the Stokes equations near each single hole, where there arise problems of type \eqref{2}.

\medskip

 In \eqref{2}, $\vv:\O_\e \to \R^d$ is the \emph{unknown}, $\bg:\O_\e\to \R^{d\times d}$ is named the \emph{source function}, $0\in T\subset \Omega\subset \R^d$  are both simply connected, bounded domains. We assume that $\O$ is of class $C^1$ (Lipschitz domain can be also considered, while the corresponding result may be weaker; see Remark \ref{rem-lip}), and $T$ is of class $C^{2,\b}$ for some $0<\b<1$. Here, and throughout the paper, $\e\in (0,1)$ is a small parameter.  Without loss of generality, we suppose
$$
T\subset B_{1/2}:=B(0,1/2) \subset B_1:=B(0,1) \subset \Omega.
$$

We remark that the $C^{2,\b}$ smoothness assumption for $\d T$ may be relaxed. However, our main concern is the asymptotic behavior of the solutions to \eqref{2} as $\e\to 0$, so we will not work in the direction of relaxing the smoothness of $\d T$.

\subsection{Some notations}
We recall some notations used in this paper.  For any domain $\O\subset \R^d$ and a vector-valued function $\vv=(\vv_1,\cdots, \vv_d):\O \to \R^d$, we define its gradient $\nabla \vv$ to be a $d\times d$ matrix of which the $(i,j)$ element is
$$
\left(\nabla \vv\right)_{i,j}=\d_{x_j}\vv_i.
$$
For a matrix-valued function $\bg=(\bg_{i,j})_{1\leq i,j\leq d}:\O\to \R^{d\times d}$, we define its divergence to be a vector as
$$
\dive \bg =\Big(\sum_{j=1}^{d}\d_{x_j} \bg_{i,j}\Big)_{1\leq i\leq d}.
$$

\smallskip

For any $1< p < \infty$, the notation $p'$ denotes the Lebesgue conjugate component of $p$: $1/p'+1/p=1$.  For any $1<p<\infty$ and any domain $\O\subset \R^d$, we use classical notation $L^p(\O)$ to denote the space of Lebesgue measurable functions such that
$$
\|u\|_{L^p(\O)}:=\Big(\int_\O |u(x)|^p\,\dx\Big)^{1/p}<\infty.
$$
We use the classical notation $W^{1,p}(\Omega)$ to denote the Sobolev space with the norm
$$
\|u\|_{W^{1,p}(\Omega)}:=\big(\|u\|_{L^p(\O)}^p+\|\nabla u\|_{L^p(\O)}^p\big)^{\frac{1}{p}},
$$
where the derivatives are defined in the sense of distribution.

The notation $W_0^{1,p}(\Omega)$ denotes the completion of $C_c^\infty(\O)$ with respect to the norm $\|\cdot\|_{W^{1,p}(\Omega)}$. The notation $W^{-1,p}(\O)$ denotes the dual space of the Sobolev space $W_0^{1,p'}(\O)$. The definition of the $W^{-1,p}(\O)$ norm is classical:
\be\label{def-W-1p}
\|u\|_{W^{-1,p}(\O)}:=\sup_{\phi \in C_c^\infty(\O),\,\|\phi\|_{W^{1,p'}=1}}|\langle u,\phi \rangle|.\nn
\ee
We use the notation $L^{p}_0(\Omega)$ to denote the space of $L^p(\Omega)$ functions with zero mean value:
\be\label{Lp0}
L^p_0(\Omega):=\Big\{f\in L^p(\Omega): \, \int_{\O} f\,\dx=0\Big\}.\nn
\ee

\subsection{Known results for Stokes equations in a bounded domain}

We remark that, throughout the paper,  a solution of the Dirichlet problem of the Stokes equations defined in a domain $D$ means a weak solution: the equations are satisfied in the sense of distribution with test functions in $C_{c}^{\infty}(D)$ and the boundary conditions are satisfied in the sense of traces of Sobolev functions. This definition of weak solutions is now widely used and well understood.

The well-posedness theory for the Dirichlet problems of the Stokes equations in bounded domains is well developed. We give a remark to collect some known results.

We refer to Theorem IV.6.1 in \cite{Galdi-book} (for domains of class $C^2$), Theorem 5.1 in \cite{DM} (for domains of class $C^1$), Section 3 of \cite{MM} (a review of results for Lipschitz and $C^1$ domains) and Theorem 2.9 in \cite{BS} (for Lipschitz domains in $\R^3$) for more details.
\begin{remark}\label{rem:stokes}
\begin{itemize}
\item[{\rm ({\bf i})}.]
For any connected, bounded domain $\O\subset \R^d,~d\geq 2$ of class $C^1$, for any $1<p<\infty$ and any $\hh\in W^{-1,p}(\O;\R^d)$, the following Dirichlet problem of the Stokes equations
\ba
-\Delta \ww+\nabla \pi= \hh,\quad \dive \ww=0, \quad \mbox{in}~\O;\qquad \ww=0,\quad \mbox{on} ~\d\O\nn
\ea
 admits a unique solution $(\ww,\pi)\in W_0^{1,p}(\O;\R^d)\times L_0^p (\O)$ and there holds the estimate
\be\label{est-classical2}
\|\ww\|_{W^{1,p}_0(\Omega;\R^d)}+ \|\pi\|_{L^{p}(\Omega)} \leq C(p,d,\O) \ \|\hh\|_{W^{-1,p}(\Omega;\R^d)}.
\ee

\item[{\rm ({\bf ii})}.]  If the bounded domain $\O$ is only of class Lipschitz, the result in {\rm ({\bf i})} holds true for $p$ near $2$:     $p_1'(d,\O) < p < p_1(d,\O)$ for some $p_1(d,\O)>2$. In particular, in three dimensions, one has $p_1(3,\O)>3$ (see {\rm Theorem 2.9 in \cite{BS}}).

\end{itemize}

\end{remark}

\subsection{Main results}

Now we state our main theorem concerning the uniform $W^{1,p}$ estimate:
\begin{theorem}\label{thm-sto} Let $d'<p<d$ ($p=2$ if $d=2$), and let $\bg\in L^p(\Omega_\e;\R^{d\times d})$. Then the unique solution $(\vv,\pi) \in W_0^{1,p}(\O_\e;\R^d)\times L_0^{p}(\O_\e)$ to \eqref{2} satisfies the uniform estimate
\be\label{est-sto}
\|\nabla \vv \|_{L^p(\Omega_\e;\R^{d\times d})}+ \|\pi\|_{L^p(\Omega_\e) } \leq C \, \|\bg\|_{L^p(\O_\e;\R^{d\times d})}
\ee
for some constant $C=C(p,d,\O,T)$ independent of $\e$.

\end{theorem}

We give two remarks about  Theorem \ref{thm-sto}:
\begin{remark}\label{rem-lip}
\begin{itemize}
\item[{\rm ({\bf i})}.]
For any fixed $\e$, if $\bg\in L^p(\Omega_\e;\R^{d\times d})$ with $1<p<\infty$, the well-posedness of \eqref{2} is known, see {\rm Remark \ref{rem:stokes}}. The key novelty of  {\rm Theorem \ref{thm-sto}} lies in obtaining the uniform estimate as $\e \to 0$.  When $p=2$, the estimate constant $C(2,d,\O_\e)=1$ in \eqref{est-classical2}, while when $p\neq 2$, the classical estimate constant $C=C(p,d,\O_\e)$ in \eqref{est-classical2} depends on the $C^1$ (or Lipschitz) character of the domain $\O_\e$, which is unbounded as $\e \to 0$.

\item[{\rm ({\bf ii})}.]
The result in {\rm Theorem \ref{thm-sto}} can be generalized to domains with Lipschitz boundary. In the case with Lipschitz domain $\O$, the choice range of $p$ obeys the restriction $$\max \{p_1',d'\}<p<\min\{p_1,d\},$$ where $p_1$ is restricted as in {\rm Remark \ref{rem:stokes} ({\bf ii})}. This is because we only have the existence of solutions in $W^{1,p}_0$ with $p_1'<p<p_1$.  In particular, we obtain the same results when $d=3$ as the $C^{1}$ domain case: {\rm Theorem \ref{thm-sto}} still holds true when $\O$ is of class Lipschitz.

\end{itemize}
\end{remark}

\medskip

The following result indicates that the choice range of $p$ in Theorem \ref{thm-sto} is sharp.

\begin{theorem}\label{thm-sto1-new}
 Let $1<p\leq d'$ or $p\geq d$ ($p\neq 2$ if $d=2$).   Then for any $0<\e<1$, there exists $\bg_\e\in L^p(\Omega_\e;\R^{d\times d})$ satisfying $\|\bg_\e\|_{L^p(\Omega_\e;\R^{d\times d})}\leq 1$ such that the solution family $(\vv_\e,\pi_\e) \in W^{1,p}_0(\Omega_\e;\R^{d})\times L^{p}_0(\O_\e) $ to \eqref{2} with source function $\bg_\e$ satisfies
\be\label{est-sto2-new}
\liminf_{\e \to 0}\|\nabla \vv_\e\|_{L^p(\Omega_\e;\R^{d\times d})} =\infty.
\ee

\end{theorem}

\subsection{Motivation and background}\label{sec:motivation}

The motivation of this paper mainly comes from the study of homogenization problems in the framework of fluid mechanics, that is the study of fluid flows in domains perforated with a large number of tiny holes, and the goal is to describe the asymptotic behavior of fluid flows (governed by Stokes or Navier-Stokes equations) as the number of holes goes to infinity and the size of the holes goes to zero simultaneously. The limit equations that describe the limit behavior of fluid flows are called {\em homogenized equations} which are defined in homogeneous domain without holes.

\medskip

Allaire in \cite{ALL-NS1} and \cite{ALL-NS2} gave a systematic study for the homogenization problems of the Stokes equations and the stationary incompressible Navier-Stokes equations. Allaire showed that the homogenized equations are determined by the ratio between the size of holes and the mutual distance of holes. In particular for three-dimensional domains perforated with holes of diameter $O(\e^{\alpha})$ with $\e$ the size of the mutual distance of holes, if $1\leq \a<3$ corresponding to the case of {\em large holes}, in the limit $\e \to 0$, the fluid is governed by the Darcy's law; if $\a>3$ corresponding to the case of {\em small holes}, the motion of the fluid does not change in the homogenization process and in the limit there arise the same Stokes equations or the same stationary incompressible Navier-Stokes equations in homogeneous domains; if $\a=3$ corresponding to the case of \emph{critical size of holes}, the limit equations are governed by the Brinkman's law --- a combination of the Darcy's law and the original equations.

\medskip

For the case $\a=1$ meaning that the size of the holes is comparable to their mutual distance, earlier than Allaire, Tartar \cite{Tartar1} recovered the Darcy's law from the homogenization of the Stokes equations.  Still for the case $\a=1$, the study is extended to more complicated equations: Mikeli\'{c} \cite{Mik} for the incompressible Navier-Stokes equations; Masmoudi \cite{Mas-Hom} for the compressible Navier-Stokes equations; Feireisl, Novotn\'y and Takahashi \cite{FNT-Hom} for the complete Navier-Stokes-Fourier euqations; the Darcy's law is recovered in all these studies.

Cases with different size of holes are also considered: Feireisl, Namlyeyeva and Ne{\v c}asov{\' a} in \cite{FeNaNe} studied the case $\a=3$ (critical size of holes) for the incompressible Navier-Stokes equations and they derived Brinkman's law in the limit; Feireisl, Diening and the author in \cite{FL1, DFL} considered the case $\a>3$ (small holes) for the stationary compressible Navier-Stokes equations and it is shown that the homogenized equations remain the same as the original ones. The above two results coincide with the study of Allaire for the Stokes equations.

\medskip

In the study of these homogenization problems, the so-called \emph{restriction operator} plays an important role. A restriction operator, roughly speaking, is a linear mapping from $W_0^{1,p}(D)$ to $W_0^{1,p}(D_\e)$ and preserves the divergence free property, where $D$ is the domain without holes and $D_\e$ is the domain perforated with holes.

When $p=2$, the construction of a restriction operator from $W_0^{1,2}(D)$ to $W_0^{1,2}(D_\e)$ is given by Tartar in Lemma 4 in \cite{Tartar1} for the case $\a=1$ and by Allaire in Section 2.2 in \cite{ALL-NS1} for general $\a\geq 1$.  One key step in Allaire's construction is to study the Dirichlet problem for the Stokes equations in a neighborhood of each hole. After rescalling and changes of variables, a Dirichlet problem of type \eqref{2} arises and the uniform estimate of type \eqref{est-sto} with $p=2$ is needed.  Precisely, the domain $\Omega_\e$ in Allaire's case is $B_1\setminus \e^{\a-1}\overline T$ where $\e^{\a-1}$ is the ratio between the diameter of holes which is $\e^\a$ and the mutual distance of holes which is $\e$. In the case with $p=2$, the uniform estimate \eqref{est-sto} is rather straightforward with estimate constant $1$.

A restriction operator from $W_0^{1,p}(D)$ to $W_0^{1,p}(D_\e)$ with general $p$ is needed in studying homogenization problems for more complicated systems, such as in \cite{Mik} concerning incompressible Navier-Stokes equations, in \cite{Mas-Hom} concerning compressible Navier-Stokes equations,  in \cite{FNT-Hom} concerning compressible Naiver-Stokes-Fourier system, and particularly in \cite{Book-Homo} concerning non-Newtonian flows. However, in all these references, only the case $\a=1$ is considered, where the model domain 
$$B_1\setminus \e^{\a-1}\overline T=B_1\setminus \overline T$$
is independent of $\e$.  Consequently, the uniform $W^{1,p}$ estimates of the solution to the Stokes equations on $B_1\setminus \overline T$ can be obtained by applying the classical results (see Remark \ref{rem:stokes}). As a consequence, when $\a=1$, such restriction operator from $W_0^{1,p}(D)$ to $W_0^{1,p}(D_\e)$ with general $p$ can be constructed (by A. Mikeli\'c, see Lemma 3.2 in \cite{Book-Homo}).

To extend the study of the homogenization problems to {\em general size} of holes, it is motivated to study the uniform $W^{1,p}$ estimates of the Dirichlet problems of elliptic equations in a domain with a shrinking hole, such as $B_1\setminus \e^{\a-1}\overline T$.

\medskip

The author started this study in \cite{Lu-Lap} by focusing on the Laplace equation with divergence source term in three-dimensional domain $B_1\setminus \e \overline T$ where $B_1\subset \R^3$ is the unit ball. We showed that if $3/2<p<3$, the $W^{1,p}$ estimate is uniform as $\e \to 0$; if $1<p<3/2$ or $3<p<\infty $, there are counterexamples indicating that the uniform $W^{1,p}$ estimate do not hold.

The results in this paper for the Stokes equations coincide with the results in \cite{Lu-Lap} for the Laplace equations. The study for the Stokes equations case is more complicated and there arises new issues and difficulties.  For example, one key step of the proof in \cite{Lu-Lap} relies on the explicit formula of the Green function of Laplace operator in the unit ball (see Lemma 3.2 in \cite{Lu-Lap}). In this paper we derive rather abstract analysis (see Lemma \ref{lem:lap-Be}) instead of using the Green functions such that we can generalize the argument in \cite{Lu-Lap} to make it be suitable for Stokes equations. Moreover, after rescaling we turn to consider the problem in domain $(\Omega/\e)\setminus \overline T$ (see Section \ref{sec:resc} in this paper and Section 2 in \cite{Lu-Lap}). We then decompose the original problems into two parts by employing a cut-off function defined near the hole: one part is defined in a bounded domain and another part is defined in an enlarging domain without hole. For Stokes equations, this decomposition will destroy the divergence free condition. In Section \ref{sec:change-var}, we employ the Bogovskii operator to remedy this problem.

The study for the Stokes equations case in this paper is essentially needed to construct a restriction operator from $W_0^{1,p}(D)$ to $W_0^{1,p}(D_\e)$ for general $p$. Furthermore, this is need in the study of homogenization problems of more complicated situations: such as the homogenization of compressible Navier-Stokes equations and the homogenization of non-Newtonian flows.
\medskip

This paper is organized as follows. In Section \ref{sec:app-hom} we exhibit applications of  Theorem \ref{thm-sto} in homogenization problems.  Section \ref{sec:sto-s} and Section \ref{sec:pf12} are devoted to the proof of Theorem \ref{thm-sto} and Theorem \ref{thm-sto1-new} respectively. Throughout this paper,  $C$ denotes a positive constant independent of $\e$ unless there is a specification. However, the value of $C$ may change from line to line.

\section{Applications in homogenization problems}\label{sec:app-hom}

In this section, we give an application of our main result Theorem \ref{thm-sto} in homogenization problems: we generalize Allaire's construction of restriction operator from $L^2$ framework to $L^p$ framework. As mentioned in Introduction, the restriction operator plays a very important role in the study of homogenization problems (one can find it in most literatures in this field). This generalization will be needed in the study of more complicated problems in this field, particularly for the case where the size of holes is much smaller than the mature distance between holes. 

\subsection{Perforated domains}\label{sec:domain}
In this section, we give a precise description for the perforated domains mentioned before. We focus on the case where the holes are periodically distributed in a bounded three dimensional domain. Let $D \subset \R^3$ be a bounded domain of class $C^{1}$, and let $\{T_{\e,k}\}$ be a family of $C^{2,\b}$ domains satisfying for any $k\in \Z^3$:
\be\label{def-holes1}
T_{\e,k}\subset  \subset B(x_k,b \e^\alpha) \subset  B(x_k, b_{1} \e)  \subset \subset \e C_k:=(0,1)^3+k
\ee
for some $x_k\in T_{\e,k} $, some $0<\de<1, \ b>0, \ b_{1}>0, \ \a\geq 1$ independent of $\e$. Moreover, we assume that each hole $T_{\e,k}$ is similar to the same \emph{model hole} $T$ rescaled by $\e^\a$:
 \be\label{T-Te}
 T_{\e,k}=\e^\a \mathcal{O}_{k}(T)+x_k,
 \ee
 where $\mathcal{O}_{k}$ is a rigid rotation. We assume $T$ is a simply connected domain of class $C^{2,\b}$ and is independent of $\e$.

The small domains $T_{\e,k}$ are called the \emph{holes} or \emph{obstacles}, the domain $D$ is the homogeneous domain without holes, and the corresponding \emph{perforated domain} $D_\e$ is defined as
\be\label{domain}
D_\e:=D \setminus  \bigcup_{k\in K_\e}\overline T_{\e,k},\quad K_\e:=\{k \ |\ \e\overline {C}_k\subset \Omega\}.\nn
\ee

We can see that in $D_\e$, the diameters and mutual distances of the holes are of size $O(\e^\a)$ and $O(\e)$, respectively,  and the number of holes satisfies
\be\label{number-holes}
|K_\e|= \frac{|D|}{\e^3}\big(1+o(1)\big),\quad \mbox{as $\e\to 0$}.\nn
\ee

\subsection{Restriction operator and Bogovskii type operator}\label{sec:res-bog}

By employing the uniform estimates obtained in  Theorem \ref{thm-sto}, we can generalize the construction of restriction operator by Allaire \cite{ALL-NS1} from the $L^2$ framework to the $L^p$ framework:
\begin{theorem}\label{thm-res} Let $D$ and $D_\e$ be given in {\rm Section \ref{sec:domain}}. For any $3/2<p<3$, there exists a linear operator $R_\e$ from $W_0^{1,p}(D;\R^3)$ to $W_0^{1,p}(D_\e;\R^3)$ such that
\ba\label{pt-res}
&\uu \in W_0^{1,p}(D_\e;\R^3) \Longrightarrow R_\e (\tilde \uu)=\uu \ \mbox{in}\ D_\e,\ \mbox{where} \ \tilde \uu:=\begin{cases}\uu \ &\mbox{in}\ D_\e,\\ 0  \ &\mbox{on}\ D\setminus D_\e, \end{cases}\\
&\uu \in W_0^{1,p}(D;\R^3),\  \dive \uu =0 \ \mbox{in} \ D \Longrightarrow \dive R_\e (\uu) =0 \ \mbox{in} \ D_\e,\\
&\uu \in W_0^{1,p}(D;\R^3)\Longrightarrow \|\nabla R_\e(\uu)\|_{L^p(D_\e;\R^{3\times3})}\\
&\qquad \qquad \qquad\qquad \qquad\leq C \, \big( \|\nabla \uu\|_{L^p(D;\R^{3\times3})}+\e^{\frac{(3-p)\a-3}{p}} \|\uu\|_{L^p(D;\R^3)}\big).
\ea

\end{theorem}

\begin{remark}
We remark that the result in Theorem \ref{thm-res} generalizes the result of Allaire (see Proposition 3.4.10 in \cite{ALL-NS2}) where only the case $p=2$ is considered, and also generalizes the result of Mikeli\'c (see Lemma 3.2 in \cite{Book-Homo}) where the restriction $\a=1$ is needed.
\end{remark}

Here we find one direct corollary of this generalized restriction operator, that is to obtain some uniformly bounded Bogovskii type operators, which can be seen as the inverse of the divergence operators on perforated domains. Bogovskii type operators are typically used in the study of compressible Navier-Stokes equations, in order to improve the integrability of the pressure term.

\begin{corollary}\label{thm-bog} Let $D_\e$ be the domain defined in {\rm Section \ref{sec:domain}}. For any $3/2<p<3$, there exists a linear operator $\CalB_\e$ from $L_0^{p}(D_\e)$ to $W_0^{1,p}(D_\e;\R^3)$ such that for any $f\in L_0^{p}(D_\e)$,
\ba\label{pt-bog}
&\dive \CalB_\e(f) =f \ \mbox{in} \ D_\e,\\
&\|\CalB_\e(f)\|_{W_0^{1,p}(D_\e;\R^3)}\leq C\, \Big(1+\e^{\frac{(3-p)\a-3}{p}}\Big)\|f\|_{L^p(D_\e)},
\ea
for some constant $C$ independent of $\e$.
\end{corollary}

The proof for Theorem \ref{thm-res} and Corollary \ref{thm-bog} is given in Section \ref{sec:res-bog-proof}. About these two results, we give several remarks. The first one indicates that the estimate constant $\e^{\frac{(3-p)\a-3}{p}}$ may be optimal in the sense of capacity.
\begin{remark}\label{rem-optimal-constant}
 Let $1<p<\infty$. Here we focus on the case $d=3$. We recall that the $p$-capacity of a set $D\subset \R^d$ is defined as
$${\rm Cap}_p (D):= \inf\big\{ \| f \|_{W^{1,p}(\R^{d})}^p  : f\in W^{1,p}(\R^d), \  D\subset \{f\geq 1\}\big\}.$$
If $1<p<d$, it is known that the $p$-capacity of a ball has the following estimate (see for example \cite{GS99} or \cite{Poggesi18}): let $B(0,r)\subset \R^d$ be a ball with radius $r \in (0,1)$, then
$$
{\rm Cap}_p (B(0,r)) = C r^{d-p}, \quad \mbox{for some $C=C(p,d)$ indepedent of $r$.}
$$
Since the $p$-capacity is an outer measure, we have the following estimate for the $p$-capacity of the union of all the holes in $D_{\e}$:
\be\label{cap-holes}
{\rm Cap}_p\,\big(\bigcup_{k\in K_\e} T_{\e,k}\big)\leq \sum_{k\in K_\e} {\rm Cap}_p\,(T_{\e,k}) = C \e^{-3}\e^{\alpha(3-p)} = C \e^{\alpha(3-p)-3},
\ee
where we see the same essential quantity $\alpha(3-p)-3$ as in Theorem \ref{thm-res} or  Corollary \ref{thm-bog}. Here in \eqref{cap-holes} we only have an inequality, due to the fact that $p$-capacity does not have additivity property. So this estimate may be rough. 

\medskip

For such a coincidence between the estimate of the restriction operator or the Bogovskii type operator and the estimate of the capacity of the holes, we give a possible explanation in the following. Let  $u \in W_{0}^{1,p}(D_{\e})$. Let $\varphi \in C_{c}^{\infty}(\R^{d})$ satisfying $\varphi = 1$ in $D$. Then we have that $\varphi( 1 - u)$
is an admissible function for the definition of $p$-capacity of the union of all the holes in $D_{\e}$, that is
$$
\varphi(1-u) \in W^{1,p}(\R^{d}), \quad \varphi(1-u) = 1 \quad \mbox{on} \quad \bigcup_{k\in K_\e} \overline T_{\e,k}.
$$
Direct calculation gives
$$
\| \varphi(1-u) \|_{ W^{1,p}(\R^{d})} \leq C(\varphi)  \|u\|_{ W^{1,p}(D_{\e})},
$$
where $C(\varphi)$ is a constant depending on $\varphi$. If $\a$ is large, which means the holes are relatively well separated, the estimate \eqref{cap-holes} could be close to an equality.  Due to the fact that
$$
C(\varphi)  \|u\|_{ W^{1,p}(D_{\e})}  \geq \| \varphi(1-u) \|_{ W^{1,p}(\R^{d})} \geq {\rm Cap}_p\,\big(\bigcup_{k\in K_\e} T_{\e,k}\big),
$$
we may expect
$$
 \|u\|^{p}_{ W^{1,p}(D_{\e})}  \geq C^{-1} \e^{\alpha(3-p)-3}.
$$

\medskip

Back the restriction operator and the Bogovskii type operator obtained in Theorem \ref{thm-res} and Corollary \ref{thm-bog}, after applying the operators, we have that $R_{\e}(\vu)$ and $\CalB_\e(f)$ are both in $W_{0}^{1,p}(D_{\e})$.  Then the optimal general estimate of the operator norm may be of size $1 + \e^{\frac{(3-p)\a-3}{p}}$.

While, here we did not give a lower bound estimate for the norm of the Bogovskii type operator. We think this could be an interesting problem to think about.

\end{remark}

\begin{remark}
The existence of such an operator satisfying $\eqref{pt-bog}_1$ is classical, known as Bogovskii's operator. However, the norm of the classical Bogovskii's operator depends on the Lipschitz character of the domain, while the Lipschtiz norm of domain $D_\e$ is unbounded as $\e\to 0$. In Theorem \ref{thm-bog} we give a quantitative estimate for the operator norm as shown in $\eqref{pt-bog}_2$. In particular, if the holes are suitably small such that $(3-p)\a-3\geq0$, one has uniform estimate for this Bogovskii type operator. In \cite{DFL}, another construction of a Bogovskii type operator on perforated domains is given and the estimates for these two different constructions are the same. However, these two constructions contain different local properties and may have special advantages in different occasions.
\end{remark}

\begin{remark}
We are considering more applications of this fundamental estimate in Theorem \ref{thm-sto} and this generalized restriction operator in Theorem \ref{thm-res}, such as the homogenization of non-Newtonian flows when $\a>1$ (for which $p$--Stokes equations under the case $\a=1$ is considered by Mikeli\'c \cite{Book-Homo}) and uniform $W^{1,p}$ estimates of elliptic equations in perorated domains with $\a>1$ (motivated by the study of Masmoudi in \cite{Masmoudi-CRM} for $\a=1$).
\end{remark}

\section{Proof of Theorem \ref{thm-sto}}\label{sec:sto-s}

This section is devoted to the proof of Theorem \ref{thm-sto}.

\subsection{Reformulation}\label{sec:resc}

We reformulate the Dirichlet problems \eqref{2} by introducing the following Dirichlet problem in the rescaled domain $\wto_\e:=\O_\e/\e=(\O/\e)\setminus \overline T$:
\be\label{2-1}
\left\{\begin{aligned}
-\Delta \vv_1+\nabla \pi_1 &=\dive \bg_1,\quad &&\mbox{in}~\wto_\e,\\
\dive \vv_1&=0 ,\quad &&\mbox{in}~\wto_\e,\\
\vv_1&=0,\quad &&\mbox{on} ~\d\wto_\e.
\end{aligned}\right.
\ee

By scaling technique, it can be shown that to prove Theorem \ref{thm-sto}, it is equivalent to prove the following theorem:
\begin{theorem}\label{thm-sto-s} For any $d'<p<d$ if $d\geq 3$ or $p=2$ if $d=2$, and any $\bg_1\in L^p(\wto_\e;\R^{d\times d})$, the unique solution $(\vv_1,\pi_1) \in W_0^{1,p}(\wto_\e;\R^d)\times L_0^{p}(\wto_\e)$ to \eqref{2-1} satisfies the uniform estimate
\be\label{est-sto-s}
\|\nabla \vv_1 \|_{L^p(\wto_\e;\R^{d\times d})}+ \|\pi_1\|_{L^p(\wto_\e) } \leq C \, \|\bg_1\|_{L^p(\wto_\e;\R^{d\times d})}
\ee
for some constant $C=C(p,d,\O,T)$ independent of $\e$.

\end{theorem}

We proceed several steps to see the equivalence between Theorem \ref{thm-sto} and Theorem \ref{thm-sto-s}. We first prove Theorem \ref{thm-sto} by supposing that Theorem \ref{thm-sto-s} holds true. Let $\bg\in L^p(\O_\e;\R^{d\times d})$ with $d'<p<d$ if $d\geq 3$ and $p=2$ if $d=2$, and $(\vv,\pi)\in W_0^{1,p} (\O_\e;\R^d)\times L_0^{p} (\O_\e)$ be the unique solution to \eqref{2}. Thus, the rescaled quantities
\be\label{def-u-s}
\vv_1(\cdot):= \vv(\e \cdot),\quad \pi_1(\cdot ):= \e\, \pi(\e \cdot) \quad \bg_1(\cdot ):= \e\, \bg(\e \cdot)
\ee
satisfies \eqref{2-1}. By Theorem \ref{thm-sto-s}, there holds the uniform estimate
$$
\|\nabla \vv_1 \|_{L^p(\wto_\e;\R^{d\times d})}+ \|\pi_1\|_{L^p(\wto_\e) } \leq C \, \|\bg_1\|_{L^p(\wto_\e;\R^{d\times d})}
$$
which is equivalent to
$$
\|\nabla \vv \|_{L^p(\O_\e;\R^{d\times d})}+ \|\pi\|_{L^p(\O_\e) } \leq C \, \|\bg\|_{L^p(\O_\e;\R^{d\times d})}.
$$
This is exactly our desired estimate \eqref{est-sto}.

 The other direction of the equivalence can be proved similarly. Hence, to prove Theorem \ref{thm-sto}, it is sufficient to prove Theorem \ref{thm-sto-s}. The rest of this section is devoted to proving Theorem \ref{thm-sto-s}. Let $\bg_1$ and $(\vv_1,\pi_1)$ be as in Theorem \ref{thm-sto-s}. The goal is to show the uniform estimate \eqref{est-sto-s}. The proof for the case $p=2$ is straightforward,  so in the sequel we focus on the case $d'<p<d,\ d\geq 3$.

\subsection{Change of variables}\label{sec:change-var}
 Inspired by the idea in \cite{KS} for the study of the Stokes equations in exterior domains, we introduce the cut-off function
\be\label{cut-off}
\vp\in C_c^\infty (B_1),\quad \vp\equiv 1 \ \mbox{in}\  \overline B_{1/2}, \quad 0\leq \vp\leq 1,
\ee
and consider the decomposition
\ba\label{def-v23}
&\vv_1=\vv_2+\vv_3,\quad \vv_2:=\vp\, \vv_1,\quad \vv_3:=(1-\vp) \vv_1, \\
&\pi_1=\pi_2+\pi_3,\quad \pi_2:=\vp\, \pi_1-\langle \vp\pi_1 \rangle ,\quad \pi_3:=(1-\vp) \pi_1-\langle (1-\vp) \pi_1 \rangle,\nn
\ea
where the notation $\left<\cdot \right>$ denotes the mean value:
\be\label{def-vp-pi}
\langle \vp\pi_1 \rangle:=\frac{1}{|\wto_\e|}\int_{\wto_\e} (\vp\pi_1)(x)\, \dx,\quad \langle (1-\vp)\pi_1 \rangle:=\frac{1}{|\wto_\e|}\int_{\wto_\e} \big((1-\vp)\pi_1\big)(x)\, \dx\nn
\ee
which satisfy
$$
\langle \vp\pi_1 \rangle+\langle (1-\vp)\pi_1 \rangle =\langle \pi_1 \rangle=0.
$$

Then $(\vv_2,\pi_2)$ can be viewed as the solution of the following Dirichlet problem in a
bounded domain:
\be\label{eq-v2}
\left\{\begin{aligned}
-\Delta \vv_2+\nabla \pi_2 &=\dive (\vp\bg_1)-\hh_1,\quad &&\mbox{in}~B_1\setminus \overline T,\\
\dive \vv_2&=\vv_1\cdot \nabla\vp ,\quad &&\mbox{in}~B_1\setminus \overline T,\\
\vv_2&=0,\quad &&\mbox{on} ~\d B_1\cup \d T
\end{aligned}\right.\nn
\ee
and $(\vv_3,\pi_3)$ can be seen as the solution of the following Dirichlet problem
in a pure rescaled domain without hole:
\be\label{eq-v3}
\left\{\begin{aligned}
-\Delta \vv_3+\nabla \pi_3 &=\dive \big((1-\vp)\bg_1\big)+\hh_1,\quad &&\mbox{in}~\O/\e,\\
\dive \vv_3&=-\vv_1\cdot \nabla\vp ,\quad &&\mbox{in}~\O/\e,\\
\vv_3&=0,\quad &&\mbox{on} ~\d \O/\e,
\end{aligned}\right.\nn
\ee
where
\be\label{def-h1}
\hh_1:=\vv_1 \Delta \vp + 2 \nabla \vv_1 \nabla \vp + \bg_1 \nabla \vp-\pi_1 \nabla\vp.
\ee

However the new variables $\vv_2$ and $\vv_3$ are no longer divergence free. This can be fixed by employing the classical Bogovskii's operator which says that for any bounded  Lipschitz domain $\O\in \R^d$, there exists a linear operator ${\cal B}_\O$ from $L_0^p(\O)$ to $W_0^{1,p}(\O;\R^d)$ such that for any $f\in L_0^p(\O)$:
$$
\dive \mathcal{B}_\O (f) =f \quad \mbox{in} \ \O,\quad \|\mathcal{B}_\O (f)\|_{W_0^{1,p}(\O;\R^d)} \leq C (d,p,\O)\, \|f\|_{L^p(\O)},
$$
where the constant $C$ depends only on $p,\ d$ and the Lipschitz character of $\O$.  For the existence of such an operator, we refer to Chapter III of Galdi's book \cite{Galdi-book}.

By the property of $\vp$ in \eqref{cut-off} and the property $\dive \vv_1=0$, we have $\supp\,(\vv_1\cdot \nabla\vp) \subset B_1\setminus \overline T$ and
\ba\label{pt-bog-0}
& \int_{B_1\setminus \overline T}\vv_1\cdot \nabla\vp\,\dx = \int_{B_1\setminus \overline T}\dive(\vv_1 \vp)\,\dx= \int_{B_1}\dive(\vv_1 \vp)\,\dx =0.\nn
\ea
By using the classical Bogovskii's operator, there exists $\vv_4\in W_0^{1,p}(B_1\setminus \overline T;\R^d)$ such that
\be\label{est-v4}
\dive \vv_4=\vv_1\cdot \nabla\vp, \quad \|\vv_4\|_{W_0^{1,p}(B_1\setminus \overline T;\R^d)} \leq C\,\|\vv_1\cdot \nabla\vp\|_{L^p(B_1\setminus \overline T)}.
\ee

We then define
\be\label{def-v56}
\vv_5:=\vv_2-\vv_4,\quad \vv_6:=\vv_3+\vv_4,
\ee
which are both divergence free and solve respectively
\be\label{eq-v5}
\left\{\begin{aligned}
-\Delta \vv_5+\nabla \pi_2 &=\dive (\vp\bg_1+\nabla \vv_4)-\hh_1,\quad &&\mbox{in}~B_1\setminus \overline T,\\
\dive \vv_5&=0 ,\quad &&\mbox{in}~B_1\setminus \overline T,\\
\vv_5&=0,\quad &&\mbox{on} ~\d B_1\cup \d T
\end{aligned}\right.
\ee
and
\be\label{eq-v6}
\left\{\begin{aligned}
-\Delta \vv_6+\nabla \pi_3 &=\dive \big((1-\vp)\bg_1-\nabla \vv_4\big)+\hh_1,\quad &&\mbox{in}~\O/\e,\\
\dive \vv_6&=0,\quad &&\mbox{in}~\O/\e,\\
\vv_6&=0,\quad &&\mbox{on} ~\d \O/\e.
\end{aligned}\right.
\ee

\subsection{Dirichlet problem in bounded domain}

For the Dirichlet problem \eqref{eq-v5}, the classical theory (see Remark \ref{rem:stokes}) implies that the unique solution $(\vv_5,\pi_2)$ satisfies
\ba\label{est-v5}
&\|\nabla \vv_5\|_{L^p(B_1\setminus \overline T;\R^{d\times d})}+\|\pi_2\|_{L^p(B_1\setminus \overline T)}\\
&\quad\leq C \, \| \dive (\vp\bg_1+\nabla \vv_4)-\hh_1 \|_{W^{-1,p}(B_1\setminus \overline T;\R^d)}\\
&\quad \leq C \, \left(\| \vp\bg_1+\nabla \vv_4\|_{L^p(B_1\setminus \overline T;\R^{d\times d})}+\|\hh_1 \|_{W^{-1,p}(B_1\setminus \overline T;\R^d)}\right)\\
&\quad \leq C \, \left(\| \bg_1\|_{L^p(B_1\setminus \overline T;\R^{d\times d})}+\|\vv_1\|_{L^p(B_1\setminus \overline T;\R^{d})}+\|\pi_1 \|_{W^{-1,p}(B_1\setminus \overline T}\right),\nn
\ea
where in the last inequality we used \eqref{def-h1} and \eqref{est-v4}. Moreover, by \eqref{est-v4} and \eqref{def-v56}, we obtain
\ba\label{est-v2}
&\|\nabla \vv_2\|_{L^p(B_1\setminus \overline T;\R^{d\times d})}+\|\pi_2\|_{L^p(B_1\setminus \overline T)} \\
&\quad \leq C \, \left(\| \bg_1\|_{L^p(B_1\setminus \overline T;\R^{d\times d})}+\|\vv_1\|_{L^p(B_1\setminus \overline T;\R^{d})}+\|\pi_1 \|_{W^{-1,p}(B_1\setminus \overline T)}\right).
\ea

\subsection{Dirichlet problem in enlarging domain}

To study \eqref{eq-v6}, by observing that the first term of the right-hand side of $\eqref{eq-v6}_1$ is of divergence form and the second term is of compact support in $B_1\setminus \overline T$, we consider the decomposition $$\vv_6:=\vv_7+\vv_8,\quad \pi_3:=\pi_4+\pi_5,$$ where the two couples $(\vv_7,\pi_4)$ and $(\vv_8,\pi_5)$ in $W_0^{1,p}(\O/\e)\times L^p_0(\O/\e)$ solve respectively
\be\label{eq-v7}
\left\{\begin{aligned}
-\Delta \vv_7+\nabla \pi_4 &=\dive \big((1-\vp)\bg_1-\nabla \vv_4\big),\quad &&\mbox{in}~\O/\e,\\
\dive \vv_7&=0,\quad &&\mbox{in}~\O/\e,\\
\vv_7&=0,\quad &&\mbox{on} ~\d \O/\e,
\end{aligned}\right.
\ee
and
\be\label{eq-v8}
\left\{\begin{aligned}
-\Delta \vv_8+\nabla \pi_5 &=\hh_1,\quad &&\mbox{in}~\O/\e,\\
\dive \vv_8&=0,\quad &&\mbox{in}~\O/\e,\\
\vv_8&=0,\quad &&\mbox{on} ~\d \O/\e.
\end{aligned}\right.
\ee

\subsubsection{Dirichlet problem in enlarging domain with divergence form source term}
The Dirichlet problem \eqref{eq-v7} has a divergence form source term. We thus consider the following change of variable:
\be\label{def-v9}
\vv_9(\cdot):=\vv_7(\cdot/\e),\quad \pi_6(\cdot)=\e^{-1} \pi_4(\cdot/\e),\quad \bg_2(\cdot):=\e^{-1} \big((1-\vp)\bg_1-\nabla \vv_4\big)(\cdot/\e).\nn
\ee
Then $(\vv_9,\pi_6)\in W_0^{1,p}(\O;\R^d)\times L_0^p(\O)$ solves
\be\label{eq-v9}
\left\{\begin{aligned}
-\Delta \vv_9+\nabla \pi_6 &=\dive \bg_2,\quad &&\mbox{in}~\O,\\
\dive \vv_9&=0,\quad &&\mbox{in}~\O,\\
\vv_9&=0,\quad &&\mbox{on} ~\d \O.
\end{aligned}\right.\nn
\ee
We employ the known results (see Remark \ref{rem:stokes}) to obtain
\be\label{est-v9}
\|\nabla \vv_9\|_{L^p(\O;\R^{d\times d})}+\|\pi_6\|_{L^p(\O)} \leq C \, \| \bg_2\|_{L^p(\O;\R^{d\times d})},\nn
\ee
which is equivalent to
\be\label{est-v7}
\|\nabla \vv_7\|_{L^p(\O/\e;\R^{d\times d})}+\|\pi_4\|_{L^p(\O/\e)} \leq C \, \|(1-\vp)\bg_1-\nabla \vv_4 \|_{L^p(\O/\e;\R^{d\times d})}.
\ee

\subsubsection{Dirichlet problem in enlarging domain with compactly supported source term}
For Dirichlet problem \eqref{eq-v8}, by the property of $\vp$ in \eqref{cut-off}, we have
\be\label{supp-h}
{\rm supp}\,\hh_1 \subset (B_1\setminus \overline B_{1/2})\subset (B_1\setminus \overline T).\nn
\ee

We introduce the following lemma concerning the Dirichlet problems for Stokes equations in enlarging domains with compactly supported source terms, which may be of independent interest.
\begin{lemma}\label{lem:lap-Be}
Let $0<\e<1$ be a sufficient small parameter, $p>d'$ with $d\geq 3$ and $\bgg\in W^{-1,p}(\R^d;\R^d)$ be of compact support and its support is independent of $\e$. Let $\Omega$ be a bounded $C^1$ domain in $\R^d$. Then the unique solution $(\ww,\xi) \in W_0^{1,p}(\O/\e;\R^{d})\times L_0^{p}(\O/\e)$ to the following Dirichlet problem
\be\label{eq-w-s}
\left\{\begin{aligned}
-\Delta \ww+\nabla \xi &=\bgg,\quad &&\mbox{in}~\O/\e,\\
\dive \ww&=0,\quad &&\mbox{in}~\O/\e,\\
\ww&=0,\quad &&\mbox{on} ~\d \O/\e
\end{aligned}\right.\nn
\ee
satisfies
\be\label{est-w-s}
\|\nabla \ww\|_{L^p(\O/\e;\R^{d\times d})}+\|\xi\|_{L^p(\O/\e)} \leq C \, \|\bgg\|_{W^{-1,p}(\R^d;\R^{d})}
\ee
for some constant $C=C(p,d,\O)$ independent of $\e$.
\end{lemma}
\begin{proof}[Proof of {\rm Lemma \ref{lem:lap-Be}}]  Since $\bgg\in W^{-1,p}(\R^d;\R^d)$ is of compact support and its support is independent of $\e$, without loss of generality, we  may assume that ${\rm supp\,} \bgg \subset B_1.$

 We introduce the changes of variables
\be\label{def-w2}
\ww_1(\cdot)=\ww(\cdot/\e),\quad \xi_1(\cdot)=  \e^{-1}\xi( \cdot/\e) \quad \bgg_1(\cdot)=  \bgg (\cdot/\e)
\ee
which satisfies
\be\label{eq-w2}
\left\{\begin{aligned}
-\Delta \ww_1 + \nabla \xi_1 &=  \e^{-2}\bgg_1,\quad &&\mbox{in}~\O,\\
\dive \ww_1&=0,\quad &&\mbox{in}~\O,\\
\ww_1&=0,\quad &&\mbox{on} ~\d \O.
\end{aligned}\right.\nn
\ee

Then classical results (see Remark \ref{rem:stokes}) imply
\be\label{est-w2}
\|\nabla \ww_1\|_{L^p(\O;\R^{d\times d})} + \| \xi_1 \|_{L^{p}(\O)}\leq C\, \e^{-2}\, \| \bgg_1 \|_{W^{-1,p}(\O;\R^d)},
\ee
for some constant $C=C(p,d,\O)$ independent of $\e$.

Now we estimate the quantity on the right-hand side of \eqref{est-w2}. Let $\psi \in C_c^\infty (\O;\R^d)$ be any test function and $\chi \in C_c^\infty( B_1 )$ be a cut-off function such that $\chi=1$ on $\supp\, \bgg$. Then
\ba\label{est-h3}
\left|\langle \bgg_1,\psi  \rangle_{W^{-1,p}(\O), W_0^{1,p'}(\O)}\right| &= \left|\langle \bgg(\cdot/\e) ,\psi(\cdot)  \rangle_{W^{-1,p}(\O), W_0^{1,p'}(\O)} \right|\\
&=\e^d \left|\langle \bgg(\cdot) ,\psi(\e \cdot)  \rangle_{W^{-1,p}(\O/\e), W_0^{1,p'}(\O/\e)} \right|\\
&=\e^d \left|\langle \bgg(\cdot) ,\chi(\cdot)\psi(\e \cdot)  \rangle_{W^{-1,p}(\O/\e), W_0^{1,p'}(\O/\e)} \right|\\
&\leq \e^d \,\|\bgg\|_{W^{-1,p}(\R^d)} \| \chi(\cdot)\psi(\e \cdot) \|_{W^{1,p'}(\R^d)}.
\ea
We calculate
\ba\label{est-h3-1}
&\| \chi(\cdot)\psi(\e \cdot) \|_{W^{1,p'}(\R^d)}^{p'} = \| \chi(\cdot)\psi(\e \cdot) \|_{W^{1,p'}( B_1 )}^{p'}\\
&\quad \leq  C\,\left(\big\| | \chi(\cdot)+\nabla\chi(\cdot)|\psi(\e \cdot) \big\|_{L^{p'}( B_1 )}^{p'}+ \e^{p'} \| \chi(\cdot)(\nabla\psi)(\e \cdot) \|_{L^{p'}( B_1 )}^{p'}\right)\\
&\quad \leq C\, \left(\e^{-d} \| \psi(\cdot) \|_{L^{p'}(B_\e)}^{p'}+ \e^{p'} \e^{-d} \| \nabla\psi(\cdot) \|_{L^{p'}(B_\e)}^{p'}\right)\\
&\quad \leq C\, \e^{-d}\e^{p'} \| \nabla\psi(\cdot) \|_{L^{p'}(\O)}^{p'},
\ea
where in the last inequality we used H\"older inequality and Sobolev embedding:
\ba\label{est-h3-2}
\| \psi \|_{L^{p'}(B_\e)} &\leq C\, \e \, \| \psi \|_{L^{(p')^*}(B_\e)}
\leq C\, \e \, \| \psi \|_{L^{(p')^*}(\O)}\leq  C\, \e \, \| \nabla\psi \|_{L^{p'}(\O)},
\ea
where
$$
\frac{1}{(p')^*}=\frac{1}{p'}-\frac{1}{d}.
$$
In \eqref{est-h3-2} we also used the assumption $p>d'$, which indicates $p'<d$.

We combine \eqref{est-h3} and \eqref{est-h3-1} to obtain
\be\label{est-h3-3}
\left|\langle \bgg_1,\psi  \rangle_{W^{-1,p}(\O)\times W_0^{1,p'}(\O)}\right| \leq C\, \e^{1+d/p} \|\bgg\|_{W^{-1,p}(\R^d)} \| \nabla\psi\|_{L^{p'}(\O)},\nn
\ee
which implies
\be\label{est-h3-f}
\| \bgg_1\|_{W^{-1,p}(\O)}  \leq C\, \e^{1+d/p}\, \|\bgg\|_{W^{-1,p}(\R^d)}.\nn
\ee
Together with \eqref{est-w2} we deduce
\be\label{est-w2-1}
\|\nabla \ww_1\|_{L^p(\O;\R^{d\times d})} + \| \xi_1 \|_{L^{p}(\O;\R^d)}\leq C\, \e^{-1+d/p}\, \| \bgg \|_{W^{-1,p}(\R^d;\R^d)}.
\ee
By \eqref{def-w2}, we have
\be\label{est-w2-2}
\|\nabla \ww_1\|_{L^p(\O;\R^{d\times d})}= \e^{-1+d/p} \|\nabla \ww \|_{L^p(\O/\e;\R^{d\times d})},\quad \| \xi_1\|_{L^p(\O)}= \e^{-1+d/p} \| \xi \|_{L^p(\O/\e)}.
\ee

The estimates in \eqref{est-w2-1} and \eqref{est-w2-2} imply \eqref{est-w-s} and we complete the proof of Lemma \ref{lem:lap-Be}.

\end{proof}

We now apply Lemma \ref{lem:lap-Be} to Dirichlet problem \eqref{eq-v8} to obtain:
\be\label{est-v8}
\|\nabla \vv_8\|_{L^p(\O/\e;\R^{d\times d})}+\|\pi_5\|_{L^p(\O/\e)} \leq C \, \|\hh_1\|_{W^{-1,p}(\R^d;\R^{d})}.
\ee
From the above two estimates \eqref{est-v7} and \eqref{est-v8}, by \eqref{cut-off}, \eqref{def-h1} and \eqref{est-v4}, direct calculation gives
\ba\label{est-v78}
&\|\nabla \vv_7\|_{L^p(\O/\e;\R^{d\times d})}+\|\pi_4\|_{L^p(\O/\e)} \\
&\qquad \leq C \, \left(\|\bg_1 \|_{L^p(\wto_\e;\R^{d\times d})}+ \|\vv_1\|_{L^p(B_1\setminus \overline T;\R^d)}\right),\\
&\|\nabla \vv_8\|_{L^p(\O/\e;\R^{d\times d})}+\|\pi_5\|_{L^p(\O/\e)}\\
&\qquad \leq C \,\left( \|\bg_1 \|_{L^p(B_1\setminus \overline T;\R^{d\times d})}+  \|\vv_1\|_{L^p(B_1\setminus \overline T;\R^{d})}+\|\pi_1 \|_{W^{-1,p}(B_1\setminus \overline T)}\right).\nn
\ea
This implies
\ba\label{est-v6}
&\|\nabla \vv_6\|_{L^p(\O/\e;\R^{d\times d})}+\|\pi_3\|_{L^p(\O/\e)}\\
&\quad \leq C\, \left(\|\bg_1 \|_{L^p(\wto_\e;\R^{d\times d})}+ \|\vv_1\|_{L^p(B_1\setminus \overline T;\R^d)}+\|\pi_1 \|_{W^{-1,p}(B_1\setminus \overline T)}\right).\nn
\ea
Together with \eqref{est-v4} and \eqref{def-v56}, we finally obtain
\ba\label{est-v3}
&\|\nabla \vv_3\|_{L^p(\O/\e;\R^{d\times d})}+\|\pi_3\|_{L^p(\O/\e)} \\
&\quad\leq C \, \left(\|\bg_1 \|_{L^p(\wto_\e;\R^{d\times d})}+ \|\vv_1\|_{L^p(B_1\setminus \overline T;\R^d)}+\|\pi_1 \|_{W^{-1,p}(B_1\setminus \overline T)}\right).
\ea

\subsection{End of the proof}\label{sec:end-lap}

First of all, summing up the estimates in \eqref{est-v2} and \eqref{est-v3} implies directly the following result:
\begin{proposition}\label{prop:u1}
Let $p>d'$ with $d\geq 3$, $\bg_1\in L^p(\wto_\e;\R^{d\times d})$ and $(\vv_1,\pi_1)\in W_0^{1,p}(\wto_\e;\R^d)\times L_0^{p}(\wto_\e)$ be the unique solution to \eqref{2-1}. Then there holds
\ba\label{est-v1}
&\|\nabla \vv_1\|_{L^p(\wto_\e;\R^{d\times d})}+\|\pi_1\|_{L^p(\wto_\e)} \leq C \, \left(\|\bg_1 \|_{L^p(\wto_\e;\R^{d\times d})}+ \|\vv_1\|_{L^p(B_1\setminus \overline T;\R^d)}+\|\pi_1 \|_{W^{-1,p}(B_1\setminus \overline T)}\right)\nn
\ea
for some constant $C=C(p,d,\O,T)$ independent of $\e$.
\end{proposition}

\subsubsection{Contradiction and compactness argument}

Now we employ contradiction argument to  to prove Theorem \ref{thm-sto-s}. By contradiction we suppose there exists $p\in (d',d)$, a number sequence $\{\e_k\}_{k\in\N}\subset (0,1)$ and a function sequence $\{\bg_k\}_{k\in \N}\subset L^{p}(\wto_{\e_k};\R^{d\times d})$ such that
\be\label{ek-gk}
\e_k \to 0\ \mbox{as $k\to \infty$},\quad \|\bg_k\|_{L^{p}(\wto_{\e_k};\R^{d\times d})}\to 0 \ \mbox{as $k\to \infty$},
\ee
such that the unique solution $(\vv_k,\pi_k) \in W_0^{1,p}(\wto_{\e_k};\R^d)\times L_0^{p}(\wto_{\e_k})$ to the Dirichlet problem
 \ba\label{eq-vk}
 \left\{\begin{aligned}
-\Delta \vv_k+\nabla \pi_k &=\dive \bg_k,\quad &&\mbox{in}~\wto_{\e_k},\\
\dive \vv_k&=0 ,\quad &&\mbox{in}~\wto_{\e_k},\\
\vv_k&=0,\quad &&\mbox{on} ~\d\wto_{\e_k}
\end{aligned}\right.
\ea
satisfies
\ba\label{est-vk}
\|\nabla \vv_k\|_{L^{p}(\wto_{\e_k};\R^{d\times d})}+\|\pi_k\|_{L^{p}(\wto_{\e_k})}=1 \quad \mbox{for any $k\in \N$}.
\ea

By Proposition \ref{prop:u1}, we have
\ba\label{est-vk2}
&\|\nabla \vv_k\|_{L^{p}(\wto_{\e_k};\R^{d\times d})}+\|\pi_k\|_{L^{p}(\wto_{\e_k})}\\
& \quad\leq  C \, \left(\|\bg_k \|_{L^p(\wto_\e;\R^{d\times d})}+ \|\vv_k\|_{L^p(B_1\setminus \overline T;\R^d)}+\|\pi_k \|_{W^{-1,p}(B_1\setminus \overline T)}\right).
\ea

Since $\vv_k \in W_0^{1,p}(\wto_{\e_k};\R^d)$, the uniform estimate assumed in \eqref{est-vk} and the Sobolev embedding theorem
implies
\be\label{est-vk1}
\sup_{k\in\N} \|\vv_k\|_{L^{p*}(\wto_{\e_k};\R^d)} \leq C\,\sup_{k\in\N} \|\nabla \vv_k\|_{L^{p}(\wto_{\e_k};\R^{d\times d})}\leq C,\quad \frac{1}{p^*}:=\frac{1}{p}-\frac{1}{d},
\ee
for some constant $C=C(p,d)$ independent of $\e$. We remark that in \eqref{est-vk1} the assumption $p<d$ is used.

\medskip

We consider the zero extensions $(\tilde \vv_k,\tilde \pi_k,\tilde \bg_k)$ defined as:
\ba\label{t-vk}
&\tilde \vv_k= \vv_k \ \mbox{in}\ \wto_{\e_k},\quad \tilde \vv_k=0\  \mbox{in} \ \R^d\setminus  \wto_{\e_k},\\
&\tilde \pi_k= \pi_k \ \mbox{in}\ \wto_{\e_k},\quad \tilde \pi_k=0\  \mbox{in} \ \R^d\setminus  \wto_{\e_k},\\
&\tilde \bg_k= \bg_k \ \mbox{in}\ \wto_{\e_k},\quad \tilde \bg_k=0\  \mbox{in} \ \R^d\setminus  \wto_{\e_k}.\nn
\ea

Since $\vv_k \in W_0^{1,p}(\wto_{\e_k};\R^d)$,  we have
 \be\label{t-wk2}
 \nabla \tilde \vv_k=\nabla \vv_k \ \mbox{in}\ \wto_{\e_k},\quad \nabla \tilde \vv_k=0\  \mbox{in} \ \R^d\setminus  \wto_{\e_k}.\nn
 \ee

By \eqref{ek-gk}, \eqref{est-vk} and \eqref{est-vk1}, we have the uniform estimates for the extensions:
\ba\label{est-t-vk}
&\|\tilde \bg_k\|_{L^{p}(\R^d\setminus \overline T;\R^{d\times d})}=\| \bg_k\|_{L^{p}(\wto_{\e_k};\R^{d\times d})}\to 0,\ \mbox{as $k\to \infty$},\\
&\|\nabla \tilde \vv_k\|_{L^{p}(\R^d\setminus \overline T;\R^{d\times d})}+\|\tilde \pi_k\|_{L^{p}(\R^d\setminus \overline T)}=1,\quad\sup_{k\in\N} \|\tilde \vv_k\|_{L^{p*}(\R^d\setminus \overline T;\R^d)} \leq C.\nn
\ea
Furthermore, we have  the weak convergence up to a substraction of subsequence:
\ba\label{conv-t-vk}
&\tilde \vv_k\to \vv_\infty \ \mbox{weakly in $L^{p*}(\R^d\setminus \overline T;\R^d)$}, \quad \nabla \tilde \vv_k \to \nabla \vv_\infty \ \mbox{weakly in $L^{p}(\R^d\setminus \overline T;\R^{d\times d})$},\\
&\tilde \pi_k \to \pi_\infty \ \mbox{weakly in $L^{p}(\R^d\setminus \overline T)$}.
\ea
In the following, we will show the weak limit couple $(\vv_\infty,\pi_\infty)$ is zero.

\smallskip

Since $\dive \tilde \vv_k =0$ for any $k$, we have
\be\label{dive-tv=0}
\dive \vv_\infty =0.
\ee

For any $\psi \in C_c^\infty(\R^d\setminus \overline T;\R^d)$, there exists $k_0$ such that for all $k\geq k_0$, we have $\psi \in C_c^\infty(\wto_{\e_k};\R^d)$. Then $\psi$ is a good test function for \eqref{eq-vk} when $k\geq k_0$. This implies that for any $k\geq k_0$,
\be\label{eq-tvk1}
\int_{\wto_{\e_k}} \nabla \vv_k :\nabla \psi \,\dx - \int_{\wto_{\e_k}} \pi_k\, \dive \psi \,\dx = - \int_{\wto_{\e_k}}  \bg_k :\nabla \psi \,\dx.\nn
\ee
This is equivalent to
\be\label{eq-tvk2}
\int_{\R^d \setminus \overline T} \nabla \tilde \vv_k :\nabla \psi \,\dx - \int_{\R^d \setminus \overline T} \tilde \pi_k \,\dive \psi \,\dx = - \int_{\R^d \setminus \overline T}  \tilde \bg_k :\nabla \psi \,\dx.
\ee

By \eqref{ek-gk} and \eqref{conv-t-vk}, passing $k\to \infty$ in \eqref{eq-tvk2} implies
\be\label{eq-tvl}
\int_{\R^d \setminus \overline T} \nabla  \vv_\infty :\nabla \psi \,\dx - \int_{\R^d \setminus \overline T}  \pi_\infty \dive \psi \,\dx = 0.
\ee

Then by \eqref{dive-tv=0} and \eqref{eq-tvl}, we deduce
\be\label{eq-tvf}
 \left\{\begin{aligned}
-\Delta \vv_\infty+\nabla \pi_\infty &=0,\quad &&\mbox{in}~\R^d \setminus \overline T,\\
\dive \vv_\infty &=0 ,\quad &&\mbox{in}~\R^d \setminus \overline T.
\end{aligned}\right.
\ee

In the sequel of this section, we will show that $\vv_\infty$ belongs to some homogeneous Sobolev space that falls into the framework in \cite{KS} concerning the Stokes equations in exterior domains. Then we employ the result in \cite{KS} to deduce $\vv_\infty=0,\ \pi_{\infty}=0$.

\subsubsection{Homogeneous Sobolev spaces}\label{sec:homo-sobolev}

We now recall some concepts of the homogeneous Sobolev spaces. The materials are mainly taken from Chapter II.6 and II.7 of Galdi's book \cite{Galdi-book}. Let $1\leq q <\infty$, $\Lambda :=\R^d\setminus \bar\o$ be an exterior domain with $\o$ a bounded Lipschitz domain in $\R^d$. We define the linear space
\be\label{def-D1q}
D^{1,q}(\L)=\{u\in L_{loc}^1(\L)\,:\, | u |_{D^{1,q}(\L)} <\infty \},\quad | u |_{D^{1,q}(\L)}:=\|\nabla u\|_{L^q(\L)}.
\ee
The space $D^{1,q}$ is generally not a Banach space. However, if we introduce the equivalent classes for any $u\in D^{1,q}(\L)$:
$$
[u] = \{u +c,\ c \in \R \ \mbox{is a constant}\},
$$
the space $\dot D^{1,q}(\L)$ of all equivalence classes $[u]$  equipped with the norm
$$
\left\|[u]\right\|_{\dot D^{1,q}(\L)} := |u|_{D^{1,q}(\L)} =\|\nabla u\|_{L^q(\L)}
$$
is a Banach space.

The functional $|\cdot |_{D^{1,q}(\L)}$ introduced in \eqref{def-D1q} defines a norm in $C_c^\infty(\L)$. We introduce the Banach space $D_0^{1,q}(\L)$ which is the completion of $C_c^\infty(\L)$ with respect to the norm $|\cdot |_{D^{1,q}(\L)}$.

By Sobolev embedding theorem, if $q<d$,  for any $u\in D_0^{1,q}(\L)$ there holds  $u \in L^{q*}(\L)$ where $ q*:= \frac{dq}{d-q}$. If $q<d$, Galdi \cite[equation (II.7.14)]{Galdi-book} gave an equivalent description for $D_0^{1,q}(\L)$:
\be\label{def-tD1q}
 D_0^{1,q}(\L)=\big\{ u\in D^{1,q}(\L): \ u\in L^{q*}(\L)\  \mbox{such that $\psi u \in W_0^{1,q}(\L)$ for any $\psi \in C_c^\infty (\R^d)$}\big\},
\ee
with the equivalent norm
$$
\|\cdot \|_{D_0^{1,q}(\L)}:= | \cdot |_{D^{1,q}(\L)}+ \|\cdot \|_{L^{q*}(\L)}.
$$

\subsubsection{Derivation of contradiction}

We first prove that $\vv_\infty$ obtained in \eqref{conv-t-vk} is in the homogeneous Sobolev space $D_0^{1,p} (\R^d\setminus \overline T)$. Since $p<d$, then by  \eqref{conv-t-vk} and \eqref{def-tD1q}, it suffices to show
$$
\psi \vv_\infty \in W_0^{1,p}(\R^d\setminus \overline T ),\ \mbox{for any $\psi \in C_c^\infty (\R^d)$}.
$$
This is rather direct. Since $\tilde \vv_k$ has zero trace on $\d T$, this implies
$$
\psi \tilde \vv_k \in W_0^{1,p}(\R^d\setminus \overline T ),\ \mbox{for any $\psi \in C_c^\infty (\R^d)$}.
$$
Then  $\psi \vv_\infty $, as the weak limit of $\psi \tilde \vv_k$ in $W_0^{1,q}(\R^d\setminus \overline T )$, is necessarily in $W_0^{1,q}(\R^d\setminus \overline T )$.

\smallskip

Thus, the known result for the Dirichlet problems of Stokes equations in exterior domains (see for instance \cite[Theorem 1]{KS}) states that the Dirichlet problem \eqref{eq-tvf} admits a unique solution $(\vv_\infty,\pi_\infty)$ in $D_0^{1,p}(\R^d\setminus \overline T;\R^d)\times L^p(\R^d\setminus \overline T)$;  necessarily this unique solution is
\be\label{vl=pil=0}
\vv_\infty=0,\quad \pi_\infty=0.
\ee
We remark that, to apply the result in \cite{KS}, we need $T$ to be of class $C^{2,\b}$.  This is one reason that we assume $C^{2,\b}$ regularity of domain $T$.

\medskip

Now we are ready to derive a contradiction. By \eqref{est-vk} and \eqref{est-vk1}, we have
$$
\sup_{k\in \N} \|\vv_k\|_{W^{1,p}(B_1\setminus \overline T)} \leq C<\infty.
$$
Then the Rellich-Kondrachov compact embedding theorem implies, up to a substraction of subsequence, that
\be\label{st-con1}
\vv_k \to  \vv_\infty \quad  \mbox{strongly in}\quad L^p(B_1\setminus \overline T),\quad \pi_k \to  \pi_\infty \quad  \mbox{strongly in}\quad W^{-1,p}(B_1\setminus \overline T).\nn
\ee

Finally,  by \eqref{vl=pil=0} and \eqref{ek-gk}, we pass $k\to \infty$ in \eqref{est-vk2} to obtain a contradiction:
$$
1 \leq 0.
$$

 This implies that the uniform estimate \eqref{est-sto-s} is true and we complete the proof of Theorem \ref{thm-sto-s}.

\section{Proof of Theorem \ref{thm-sto1-new}}\label{sec:pf12}

The proof of Theorem \ref{thm-sto1-new} is done in the following subsections where different values of $p$ are considered.

\subsection{The case $p>d$}\label{sec:p>d}

In this case, we have the following theorem:
\begin{theorem}\label{thm-sto1}
 Let $p>d$ and  let $\{\bg_\e\}_{0<\e<1} \subset L^p(\O_\e;\R^{d\times d})$ be uniformly bounded: $$\sup_{0<\e<1}\| \bg_\e\|_{L^p(\O_\e;\R^{d\times d})}\leq 1.$$
     Then, up to a substraction of subsequence, the zero extensions $\tilde \bg_\e$  of $\bg_\e$ in $\O$ admit the weak limit:
$$
\tilde \bg_\e \to \bg \quad \mbox{weakly in} \ L^p(\O;\R^{d\times d}), \quad \mbox{as $\e\to 0$}.
$$
 If the unique solution $\vv_\O \in W^{1,p}(\O;\R^{d})\subset C^{0,1-\frac{d}{p}}(\O;\R^{d})$ to the following Dirichlet problem
\be\label{2-O}
-\Delta \vv_\O+\nabla \pi_\O= \dive \bg,\quad \dive \vv_\O=0, \quad \mbox{in}~\O;\qquad \vv_\O=0,\quad \mbox{on} ~\d\O
\ee
 satisfies
\be\label{v0}
\vv_\O(0)\neq 0,
\ee
then the solutions $(\vv_\e,\pi_\e) \in W^{1,p}_0(\O_\e;\R^d)\times L^{p}_0(\O_\e)$ to \eqref{2} with source functions $\bg_\e$ satisfy
\be\label{est-sto1}
\liminf_{\e \to 0}\|\nabla \vv_\e\|_{L^p(\Omega_\e;\R^{d\times d})} =\infty.
\ee

\end{theorem}

In Theorem \ref{thm-sto1}  we used the Sobolev embedding $\vv_\O \in W^{1,p}(\O;\R^{d})\subset C^{0,1-\frac{d}{p}}(\O;\R^{d})$ with $p>d$. Then it makes sense to consider the value of $\vv_{\O}$ at a point.

Theorem \ref{thm-sto1} indicates that the uniform $W^{1,p}$ estimate holds with $p>d$ necessarily if the solution $\vv_{\O}$ to \eqref{2-O} vanishes at origin. However, this property is not necessarily satisfied for a general (and smooth) $\bg$. Hence, when $p>d$, Theorem \ref{thm-sto1-new} is a direct consequence of Theorem \ref{thm-sto1}.

%
%

\begin{proof}[Proof of Theorem \ref{thm-sto1}]

 Let $p>d$ and $\{\bg_\e\}_{0<\e<1} \subset L^p(\O_\e;\R^{d\times d})$ satisfying the assumptions in Theorem \ref{thm-sto1}. Let $(\vv_\e,\pi_\e) \in W^{1,p}_0(\O_\e;\R^d)\times L^{p}_0(\O_\e)$ be the unique solution to \eqref{2} with source function $\bg_\e$. This means for any $\vp \in C_{c,{\rm div}}^\infty (\O_\e;\R^d)$ and any $\phi\in C_c^\infty(\O_\e)$, there holds
\ba\label{wk-fm-tve}
\int_{\O_\e}\nabla  \vv_{\e} : \nabla \vp \, \dx =-\int_{\O_{\e}}\bg_\e : \nabla { \vp} \, \dx,\quad \int_{\O_\e} \vv_{\e} \cdot \nabla \phi \, \dx=0.
\ea
The subscript div means divergence free:
$$
C_{c,{\rm div}}^\infty(D;\R^d):=\{\phi\in C_{c}^\infty(D;\R^d): \dive \phi=0\},\quad \mbox{for any domain $D\subset \R^d$}.
$$

By contradiction we assume
\be\label{est-sto-con}
\liminf_{\e \to 0}\|\nabla \vv_\e\|_{L^p(\Omega_\e;\R^{d\times d})} <\infty.\nn
\ee
Then there exists a subsequence $\{\e_k\}_{k\in\N}$ such that
\be\label{est-sto-con1}
\sup_{k\in \N}\|\nabla \vv_{\e_k}\|_{L^p(\Omega_\e;\R^{d\times d})}<\infty.\nn
\ee
This implies the weak convergence of the zero extensions:
 \be\label{est-sto-con2}
\tilde \vv_{\e_k} \to \tilde \vv \ \mbox{weakly in} \ W_0^{1,p}(\O;\R^d),\quad \mbox{as $k\to \infty$},\nn
\ee
where $\tilde\vv_{\e_k}$ are defined as
$$
\tilde\vv_{\e_k} =\vv_{\e_k} \quad \mbox{in}\ \O_\e,\quad \tilde\vv_{\e_k} =0 \quad \mbox{on}\ \e \overline T.
$$
Since $p>d$, the Sobolev compact embedding theorem implies that $\tilde \vv \in C^{0,1-\frac{d}{p}}(\O;\R^d)$, and, up to a substraction of subsequence, that
\be\label{st-tv}
\mbox{for any $\l<1-\frac{d}{p}$}, \quad \tilde \vv_{\e_k} \to \tilde \vv \ \mbox{strongly in}\  C^{0,\l}( \O;\R^d), \  \mbox{as} \ k\to \infty.
\ee
Hence, the fact $\tilde \vv_{\e_k}=0$ on $\e_k T \ni 0$ and the strong convergence in \eqref{st-tv} implies
\be\label{tv-0}
\tilde\vv (0)=0.
\ee

Passing $\e_k\to 0$ in \eqref{wk-fm-tve} gives
\ba\label{wk-fm-tve1}
&\int_{\O}\nabla  \tilde \vv : \nabla  \vp \, \dx =-\int_{\O}\bg : \nabla \vp \, \dx,\quad &&\mbox{for any $\vp \in C_{c,{\rm div}}^\infty (\O\setminus \{0\};\R^d)$},\\
&\int_{\O}  \tilde \vv \cdot \nabla  \phi \, \dx =0, \quad &&\mbox{for any $\phi \in C_{c}^\infty (\O\setminus \{0\})$}.
\ea
Thus, by equations \eqref{tv-0} and \eqref{wk-fm-tve1}, we obtain
\be\label{eq-tv}
-\Delta \tilde \vv+\nabla \tilde \pi= \dive \bg,\quad \dive \tilde \vv=0, \quad \mbox{in}~\O\setminus \{0\};\qquad \tilde \vv=0,\quad \mbox{on} ~\d\O\cup \{0\},
\ee
for some $\tilde\pi\in \mathcal{D}'(\O\setminus \{0\})$ with zero mean value. Moreover, the first equation in \eqref{eq-tv} implies
\be\label{est-nabla-p}
\nabla \tilde \pi \in W^{-1,p}(\O\setminus \{0\};\R^d).\nn
\ee

Since $\O\setminus \{0\}$ is bounded and is a finite union of star-shaped domains, by employing Theorem III.3.1 in \cite{Galdi-book},  for any $1<q<\infty$, there exists a linear operator ${\mathcal B}$ (known as Bogovskii's operator) from $L^q_0(\O\setminus \{0\})$ to $W^{1,q}_0(\O\setminus \{0\};\R^d)$ such that for any $f\in L^q_0(\O\setminus \{0\})$, there holds
\be\label{bog-O-0}
\dive \CalB(f)=f \quad \mbox{in} \ \O\setminus \{0\};\quad \|\CalB(f)\|_{W^{1,q}_0(\O\setminus \{0\};\R^d)} \leq C(q,d,\O)\,\|f\|_{L^q(\O\setminus \{0\})}.\nn
\ee

Another point of view is that, $\O\setminus \{0\}$ is a John Domain on which a uniformly bounded Bogovskii type operator can be constructed (see for instance \cite{ADM}).

Then for any nonzero $f\in C_c^\infty(\O\setminus \{0\})$, we have
\ba\label{est-Lp-pi}
&\langle \tilde \pi, f\rangle = \langle \tilde \pi, f-\langle f \rangle \rangle  = \langle \tilde \pi, \dive {\mathcal B}(f-\langle f \rangle)\rangle = \langle \nabla \tilde \pi,  {\mathcal B}(f-\langle f \rangle)\rangle\\
&\quad  \leq  \|\nabla\tilde \pi \|_{W^{-1,p}(\O\setminus \{0\};\R^d)} \|{\mathcal B}(f-\langle f \rangle)\|_{W_0^{1,p'}(\O\setminus \{0\};\R^d)}\\
&\quad  \leq C\, \|\nabla\tilde \pi \|_{W^{-1,p}(\O\setminus \{0\};\R^d)} \|f-\langle f \rangle\|_{L^{p'}(\O\setminus \{0\})}\\
&\quad  \leq C\, \|\nabla\tilde \pi \|_{W^{-1,p}(\O\setminus \{0\};\R^d)} \|f\|_{L^{p'}(\O\setminus \{0\})},
\ea
where $\langle f \rangle:= \frac{1}{|\O\setminus \{0\}}|\int_{\O\setminus \{0\}}f(x)\,\dx$ is the mean value of $f$. This implies
\be\label{est-Lp-pi1}
\|\tilde \pi\|_{L^p(\O\setminus \{0\})} \leq C\,  \|\nabla \tilde \pi\|_{W^{-1,p}(\O\setminus \{0\};\R^d)}<\infty.\nn
\ee

Now we have $(\tilde \vv,\tilde \pi)$ is the solution to Dirichlet problem \eqref{eq-tv} defined in domain $\O\setminus \{0\}$.  In the following proposition, we show $(\tilde \vv,\tilde \pi)$ is also the solution to Dirichlet problem \eqref{2-O} in domain $\O$:

\begin{proposition}\label{prop:tu} The couple $(\tilde \vv,\tilde \pi)$  solves the following Dirichlet problem:
\be\label{eq-tv-f}
-\Delta \tilde \vv+\nabla \tilde \pi= \dive \bg,\quad \dive \tilde \vv=0, \quad \mbox{in}~\O;\qquad \tilde \vv=0,\quad \mbox{on} ~\d\O.\nn
\ee

\end{proposition}

\begin{proof}[Proof of Proposition \ref{prop:tu}]
It is sufficient to show the following two equalities:
\be\label{wk-fm-tu}
\int_{\O}\nabla\tilde \vv : \nabla \vp \, \dx - \int_{\O} \tilde \pi  \dive \vp \, \dx + \int_{\O}\bg : \nabla \vp=0 \, \dx \quad \mbox{for any $\vp\in C_c^\infty(\O;\R^d)$},
\ee
and
\be\label{wk-fm-tu-0}
\int_{\O} \vv \cdot \nabla \phi \, \dx =0, \quad \mbox{for any $\phi\in C_c^\infty(\O)$}.
\ee

To this end, we introduce cut-off functions sequence $\{\phi_n\}_{n\in \Z_+}\subset C^\infty(\R^d)$ satisfying
\be\label{cut-off2}
0\leq \phi_n\leq 1,\quad \phi_n =0 \ \mbox{in}\ B_{1/n}, \quad \phi_n =1 \ \mbox{on}\ B_{2/n}^c,\quad |\nabla\phi_n|\leq 4\,n.\nn
\ee
Then for any $1<q<\infty$, direct calculation gives
\be\label{cut-off20}
\|(1-\phi_n)\|_{L^q(\R^d)} \leq C \, n^{-\frac{d}{q}},\quad \|\nabla \phi_n\|_{L^q(\R^d;\R^d)} \leq C\, n^{1-\frac{d}{q}}.\nn
\ee

Let $\vp\in C_c^\infty(\O;\R^d)$ be a test function. We calculate
\ba\label{wk-fm-tu2}
&\int_{\O} (\nabla \tilde \vv +\bg) : \nabla \vp \, \dx = \int_{\O} (\nabla \tilde \vv+\bg) : \nabla (\vp\phi_n) \, \dx \\
&\quad - \int_{\O} (\nabla\tilde \vv + \bg)  : (\vp \otimes \nabla \phi_n) \, \dx + \int_{\O} (\nabla\tilde \vv+ \bg)  : (1-\phi_n)\nabla \vp \, \dx,\\
&\int_{\O} \tilde \pi \, \dive \vp \, \dx  = \int_{\O} \tilde \pi\,  \dive (\vp \phi_n) \, \dx -  \int_{\O} \tilde \pi  \vp\cdot \nabla \phi_n \, \dx + \int_{\O} \tilde \pi  (\dive \vp)(1- \phi_n) \, \dx.
\ea

Since $\phi_n \vp \in C_c^\infty(\O\setminus \{0\};\R^d)$ and $(\tilde \vv,\tilde \pi)$ is the solution to \eqref{eq-tv}, we then have
\be\label{est-tvn0}
\int_{\O} (\tilde \vv+\bg) : \nabla (\vp\phi_n) \, \dx -\int_{\O} \tilde \pi \, \dive (\vp \phi_n) \, \dx =0,\quad \mbox{for any $n\in\Z_+$}.
\ee
By \eqref{wk-fm-tu2} and \eqref{est-tvn0}, we obtain
\ba\label{est-tvn}
&\int_{\O}\nabla\tilde \vv : \nabla \vp \, \dx - \int_{\O} \tilde \pi  \dive \vp \, \dx + \int_{\O}\bg : \nabla \vp \, \dx \\
&\quad = - \int_{\O} (\nabla\tilde \vv + \bg)  : (\vp \otimes \nabla \phi_n) \, \dx + \int_{\O} (\nabla\tilde \vv+ \bg)  : (1-\phi_n)\nabla \vp \, \dx \\
&\qquad + \int_{\O} \tilde \pi  \vp\cdot \nabla \phi_n \, \dx - \int_{\O} \tilde \pi  (\dive \vp)(1- \phi_n) \, \dx.
\ea

 We then calculate
\ba\label{wk-fm-tu3}
&\left| \int_{\O} (\nabla\tilde \vv + \bg)  : (\vp \otimes \nabla \phi_n) \, \dx \right| \leq C\, \|\nabla\tilde \vv +\bg\|_{L^p(\O;\R^{d\times d})}\,\|\nabla \phi_n\|_{L^{p'}(\O;\R^d)} \leq C \, n^{1-\frac{d}{p'}},\\
&\left|\int_{\O} (\nabla\tilde \vv +\bg ) : (1-\phi_n)\nabla \vp \, \dx\right|\leq C\, \|\nabla\tilde \vv +\bg\|_{L^p(\O;\R^{d\times d})}\,\|(1- \phi_n)\|_{L^{p'}(\O)} \leq C \,n^{-\frac{d}{p'}},\\
&\left|\int_{\O} \tilde \pi  \vp\cdot \nabla \phi_n \, \dx\right|\leq C \|\tilde \pi\|_{L^p(\O)} \|\nabla \phi_n\|_{L^{p'}(\O;\R^d)} \leq C\,n^{1-\frac{d}{p'}},\\
&\left|\int_{\O} \tilde \pi  (\dive \vp)(1- \phi_n) \, \dx\right| \leq C \|\tilde \pi\|_{L^p(\O)} \|(1- \phi_n)\|_{L^{p'}(\O)} \leq C\,n^{-\frac{d}{p'}}.\nn
\ea

Since $p>d\geq 3$, the power $1-d/p'<0$. Therefore,  passing $n\to \infty$ in \eqref{est-tvn} implies the result in \eqref{wk-fm-tu}. Similar argument gives \eqref{wk-fm-tu-0}. We thus complete the proof of Proposition \ref{prop:tu}.

\end{proof}

Hence, by the uniqueness of the solution to Dirichlet problem  \eqref{2-O}, we have $\tilde \vv=\vv_\O$. We thus obtain a contradiction between \eqref{v0} and \eqref{tv-0}.  Then the estimate \eqref{est-sto1} holds true and we finish the proof of Theorem \ref{thm-sto1}.

\end{proof}

\subsection{The case $1<p<d'$} \label{sec:1<p<d'}

By Theorems \ref{thm-sto1} and a dual argument, we can prove the following result:

\begin{theorem}\label{thm-sto2}
 Let $1<p < d'$. If there exists $\bg\in L^{p'}(\O;\R^{d\times d})$ such that \eqref{2-O}-\eqref{v0} are satisfied, then for any $0<\e<1$ there exists ${\bf H}_\e\in L^p(\Omega_\e;\R^{d\times d})$ satisfying $\|{\bf H}_\e\|_{L^p(\Omega_\e;\R^{d\times d})}=1$ such that the solutions $(\ww_\e,\xi_\e) \in W^{1,p}_0(\Omega_\e;\R^{d})\times L^{p}_0(\O_\e) $ to \eqref{2} with source functions ${\bf H}_\e$ satisfy
\be\label{est-sto2}
\liminf_{\e \to 0}\|\nabla \ww_\e\|_{L^p(\Omega_\e;\R^{d\times d})} =\infty.
\ee

\end{theorem}

\begin{proof}[Proof of Theorem \ref{thm-sto2}]
Let $p$  and $\bg\in L^{p'}(\O;\R^d)$ be as in Theorem \ref{thm-sto2}. This ensures that the solution $\vv_\e$ to \eqref{2} with source function $\bg$ satisfies the estimate \eqref{est-sto1}.  Let
\be\label{def-He}
{\bf H}_\e:=\frac{|\nabla \vv_\e|^{p'-2}\nabla \vv_\e}{\|\nabla \vv_\e\|_{L^{p'}(\O_\e;\R^{d\times d})}^{\frac{p'}{p}}}\nn
\ee
satisfying
$
\|{\bf H}_\e\|_{L^p(\O_\e;\R^{d\times d})}=1,
$
and $(\ww_\e,\xi_\e) \in W^{1,p}_0(\Omega_\e)\times L^{p}_0(\O_\e) $ be the unique solution to \eqref{2} with source function ${\bf H}_\e$.  Then
\ba\label{est-we1}
\|\nabla \ww_\e\|_{L^{p}(\Omega_\e;\R^{d\times d})} & \geq \|\bg\|_{L^{p'}(\Omega_\e;\R^{d\times d})}^{-1}|\langle \nabla \ww_\e,\bg\rangle|=\|\bg\|_{L^{p'}(\Omega_\e;\R^{d\times d})}^{-1}|\langle  \nabla \ww_\e, \nabla \vv_\e \rangle|\\
&=\|\bg\|_{L^{p'}(\Omega_\e;\R^{d\times d})}^{-1}|\langle  {\bf H}_\e, \nabla \vv_\e \rangle|=\|\bg\|_{L^{p'}(\Omega_\e;\R^{d\times d})}^{-1} \|\nabla \vv_\e\|_{L^{p'}(\O_\e;\R^{d\times d})}.\nn
\ea
Together with  \eqref{est-sto1}, we deduce our desire estimate \eqref{est-sto2}. The proof of Theorem \ref{thm-sto2} is competed.
\end{proof}

Now it is left to consider the case $p=d$ and $p=d'$. This is done in the next subsection.

\subsection{The case $p\in \{d',d\}$} \label{sec:p=d-d'}

We start with the case $p=d'$ with $d\geq 3$. Clearly $1<p=d'<d$.  Let $\bg_\e\in L^{p}(\Omega_\e;\R^{d\times d})$ satisfying $\|\bg_\e\|_{L^p(\Omega_\e;\R^{d\times d})}\leq 1$ for all $0<\e<1$. Let $(\vv_\e,\pi_\e) \in W^{1,p}_0(\Omega_\e)\times L^{p}_0(\O_\e) $ be the unique weak solution to \eqref{2} with source functions $\bg_\e$. The strategy is to derive some necessary conditions ensuring the uniform estimates and to show that such necessary conditions could fail. Similarly as in Section \ref{sec:p>d}, we assume the uniform estimates holds:
\be\label{est-sto-con-new}
\liminf_{\e \to 0}\|\nabla \vv_\e\|_{L^p(\Omega_\e;\R^{d\times d})} <\infty,
\ee
which means, there exists a subsequence $\{\e_k\}_{k\in\N}$ such that
\be\label{est-sto-con1-new}
\sup_{k\in \N}\|\nabla \vv_{\e_k}\|_{L^p(\Omega_\e;\R^{d\times d})}<\infty.
\ee

On $\O_\e$, one can construct a uniformly bounded Bogovskii type operator. One way to do this is to generalize the proof of Lemma 2.2.4 in \cite{ALL-NS1} from the $L^2$ framework to the $L^p$ framework, where an observation is that the argument in the proof of Lemma 2.2.4 in \cite{ALL-NS1} works in $L^p$ framework for any $1<p<d$. Since Sobolev embedding $W^{1,p}(\R^d) \hookrightarrow L^{p^*}(\R^d)$ is used in the proof, the argument cannot be generalized to $p\geq d$.  See also Lemma \ref{lem:Alem2} given later, in the next section.  By employing such a Bogovskii type operator and by the similar argument as  in \eqref{est-Lp-pi} and \eqref{est-Lp-pi1}, we have
\be\label{est-sto-con2-new}
\sup_{k\in \N}\|\pi_{\e_k}\|_{L^p(\Omega_\e)}<\infty.
\ee

Similar as \eqref{2-1} and \eqref{def-u-s}, we consider the rescaled domain $\wto_\e:=\O_\e/\e=(\O/\e)\setminus \overline T$ and the following change of variables:
\be\label{def-u-s-new}
\tilde \vv_\e (\cdot):= \e^{d/p} \e^{-1}\vv_\e(\e \cdot),\quad \tilde \pi (\cdot ):= \e^{d/p} \pi_\e(\e \cdot), \quad \tilde \bg_\e(\cdot ):= \e^{d/p} \, \bg_\e (\e \cdot).
\ee
Clearly,
\be\label{def-u-s-new-pt1}
(\tilde \vv_\e,\tilde \pi_\e) \in W^{1,p}_0(\wto_\e;\R^{d})\times L^{p}_0(\wto_\e), \quad \tilde \bg_\e\in L^{p}(\wto_\e;\R^{d\times d})
\ee
with norms
\ba\label{def-u-s-new-pt2}
&\|\nabla \tilde \vv_\e\|_{L^{p}(\wto_\e;\R^{d\times d})} = \|\nabla \vv_\e\|_{L^{p}(\O_\e;\R^{d\times d})},\quad  \|\nabla \tilde \pi_\e\|_{L^{p}(\wto_\e)} = \|\nabla \pi_\e\|_{L^{p}(\O_\e)}, \\
&\|\tilde \bg_\e\|_{L^{p}(\wto_\e;\R^{d\times d})} = \|\bg_\e\|_{L^{p}(\O_\e;\R^{d\times d})}.
\ea
Moreover, $(\tilde \vv_\e,\tilde \pi_\e) $ solves
\be\label{2-1-new}
\left\{\begin{aligned}
-\Delta \tilde \vv_\e + \nabla \tilde \pi_\e &=\dive \tilde \bg_\e,\quad &&\mbox{in}~\wto_\e,\\
\dive \tilde \vv_\e & =0 ,\quad &&\mbox{in}~\wto_\e,\\
\tilde \vv_\e & = 0,\quad &&\mbox{on} ~\d\wto_\e.
\end{aligned}\right.
\ee

By \eqref{est-sto-con1-new}, \eqref{est-sto-con2-new} and \eqref{def-u-s-new-pt2}, we have
$$
\sup_{k\in \N}\left(\|\nabla \tilde \vv_{\e_k}\|_{L^{p}(\wto_{\e_k};\R^{d\times d})} +  \|\nabla \tilde \pi_{\e_k}\|_{L^{p}(\wto_{\e_k})} + \|\tilde \bg_{\e_k}\|_{L^{p}(\wto_{\e_k};\R^{d\times d})}\right) <\infty.
$$
Now we can apply the argument in Section \ref{sec:end-lap} to derive that (see \eqref{conv-t-vk}), up to a substraction of subsequence, there holds
\ba\label{conv-t-vk-new}
&\tilde \vv_{\e_k} \to \vv_\infty \ \mbox{weakly in $L^{p*}(\R^d\setminus \overline T;\R^d)$}, \quad \nabla \tilde \vv_{\e_k} \to \nabla \vv_\infty \ \mbox{weakly in $L^{p}(\R^d\setminus \overline T;\R^{d\times d})$},\\
&\tilde \pi_{\e_k} \to \pi_\infty \ \mbox{weakly in $L^{p}(\R^d\setminus \overline T)$}, \quad \tilde \bg_{\e_k} \to \bg_\infty \ \mbox{weakly in $L^{p}(\R^d\setminus \overline T)$}.
\ea
Indeed,  in the proof of Theorem \ref{thm-sto-s}, the restriction $d' < p < d$ is needed in different places: the restriction $p>d'$ is needed for proving Lemma \ref{lem:lap-Be} and then proving Proposition \ref{prop:u1}, while the restriction $p<d$ is needed in the contradiction argument in Section \ref{sec:end-lap}. The contradiction argument in Section \ref{sec:end-lap} works for any $1<p<d$ which includes $p=d'$ under consideration here.

Moreover, again by the similar argument as in Section \ref{sec:end-lap}, the limit in \eqref{conv-t-vk-new} solves the Stokes equations in exterior domain $\R^d\setminus \overline T:$
\be\label{eq-tvf-new}
 \left\{\begin{aligned}
-\Delta \vv_\infty+\nabla \pi_\infty & = \dive \bg_\infty,\quad &&\mbox{in}~\R^d \setminus \overline T,\\
\dive \vv_\infty &=0 ,\quad &&\mbox{in}~\R^d \setminus \overline T.\\
\vv_\infty & \in D_0^{1,p}(\R^d \setminus \overline T;\R^d),
\end{aligned}\right.
\ee
in the sense that for any $\psi \in C_c^\infty(\R^d\setminus \overline T;\R^d)$, there holds
\be\label{eq-tvl-new}
\int_{\R^d \setminus \overline T} \nabla  \vv_\infty :\nabla \psi \,\dx - \int_{\R^d \setminus \overline T}  \pi_\infty \dive \psi \,\dx = - \int_{\R^d \setminus \overline T} \bg_\infty :\nabla \psi \,\dx.
\ee
Then, by the density argument, we can prove that \eqref{eq-tvl-new} holds for any test function $\psi\in D^{1,p'}_0(\R^d\setminus \overline T;\R^d)$. Indeed, by the definition of $D^{1,p'}_0$ (see Section \ref{sec:homo-sobolev}),  for any $\psi \in D^{1,p'}_0(\R^d\setminus \overline T;\R^d)$, there exists a sequence $\psi_k \in C_c^\infty(\R^d\setminus \overline T;\R^d), \ k \in \N$, such that
$$
\nabla \psi_k \to \nabla \psi \ \mbox{strongly in} \ L^{p'}(\R^d\setminus \overline T;\R^{d\times d}).
$$
Then passing $k\to \infty$ implies that \eqref{eq-tvl-new} holds for any test function $\psi\in D^{1,p'}_0(\R^d\setminus \overline T;\R^d)$.

Hence, the couple $(\vv_\infty, \pi_\infty)$ fulfills the definition of $p$-generalized solutions stated in Definition V.5.1 in Galdi's book \cite{Galdi-book}, where a couple $(\vv_\infty, \pi_\infty)\in D^{1,p}_0(\R^d\setminus \overline T;\R^d) \times L^p(\R^d\setminus \overline T)$ is said to be a $p$-generalized solution to \eqref{eq-tvf-new} provided  \eqref{eq-tvl-new} holds for any test function $\psi\in D^{1,p'}_0(\R^d\setminus \overline T;\R^d)$.

 We recall the following celebrated result describing the geometric structure of the kernel spaces of $q$-generalized solutions:
\begin{lemma}[Lemma V.5.1 in \cite{Galdi-book}]\label{lem:Galdi-kernel} Let $d\geq 2$ and $\O\subset \R^d$ be an exterior domain of class $C^2$. Denote by $S_q$
the subspace of $D_0^{1,q}(\O;\R^d)\times L^q(\O)$ constituted by q-generalized solutions $(v, \pi)$ to
\be\label{Stokes-exterior}
\left\{\begin{aligned}
-\Delta \vv+\nabla \pi & = 0 ,\quad &&\mbox{in}~\O,\\
\dive \vv&=0 ,\quad &&\mbox{in}~\O,\\
\vv&=0,\quad &&\mbox{on} ~\d\O.
\end{aligned}\right.
\ee
Then, if $1<q<d$ ($1<q\leq d$ for $d=2$), $S_q = \{\bf{0}\}$, while if $q\geq d$ ($q>d$ for $d=2$), ${\rm dim} S_q = d$.

\end{lemma}

As pointed out in (V.5.4) in \cite{Galdi-book}, a consequence of Lemma \ref{lem:Galdi-kernel} is that, there exists a $p$-generalized solution  $(\vv_\infty, \pi_\infty)\in D^{1,p}_0(\R^d\setminus \overline T;\R^d) \times L^p(\R^d\setminus \overline T)$ with $1<p=d'<d$ to \eqref{eq-tvf-new} only if the source term $\bg_\infty$ satisfies
\be\label{neces-cond-p=d}
\int_{\R^d\setminus \overline T} \bg_\infty :\nabla {\bf w} \,\dx = 0, \quad \mbox{for all}  \ {\bf w} \in S_{d}.
\ee
This is called  the  {\em compatibility condition} in Chapter V.5 in \cite{Galdi-book}. It is shown in \cite{Galdi-book} this compatibility condition is not only necessary, but also sufficient.  Hence, we have shown the following result:
\begin{theorem}\label{thm-sto3}
Let $p=d'$ with $d\geq 3$. Let $\bg_\e\in L^{p}(\Omega_\e;\R^{d\times d})$ satisfying $\|\bg_\e\|_{L^p(\Omega_\e;\R^{d\times d})}\leq 1$ for all $0<\e<1$. Let $(\vv_\e,\pi_\e) \in W^{1,p}_0(\Omega_\e;\R^{d})\times L^{p}_0(\O_\e) $ be the unique weak solution to \eqref{2} with source function $\bg_\e$.  Let  $\bg_{\infty}$ be defined through \eqref{def-u-s-new} and \eqref{conv-t-vk-new}.
If the uniform estimate \eqref{est-sto-con-new} holds,  then the compatibility condition \eqref{neces-cond-p=d} is satisfied.

\end{theorem}

While, as shown below by an example,  the compatibility condition \eqref{neces-cond-p=d} may fail for some $\bg_{\infty}$.  We thus can find a family  $\{\bg_{\e}\}_{0<\e<1}$ which fulfills our request in Theorem \ref{thm-sto1-new}.

  From Lemma \ref{lem:Galdi-kernel}, we know that $S_{d}$ is a subspace of $D_{0}^{1,d}(\R^{d}\setminus \overline T;\R^{d})\times L^{d}(\R^{d}\setminus \overline T)$ of dimension $d$. Let $\ww \in S_{d}$ be nonzero and define
$$
\tilde \bg_{\e} = \frac{|\nabla \ww|^{d-2}\nabla \ww}{\|\nabla \ww\|_{L^{d}(\R^{d}\setminus \overline T)}^{\frac{d}{d'}}}, \quad \forall \e\in (0,1).
$$
Clearly  $\tilde \bg_\e\in L^{d'}(\R^{d}\setminus \overline T;\R^{d\times d})$ satisfying $\|\tilde \bg_\e\|_{L^{d'}(\R^{d}\setminus \overline T;\R^{d\times d})} = 1$ for each $0<\e<1$.  A choice of  the family $\{\bg_{\e}\}_{0<\e<1}$ to prove Theorem \ref{thm-sto1-new} is the following:
$$
\bg_{\e} (\cdot) : =  \e^{-d/p} \,  \tilde \bg_{\e} (\cdot/\e).
$$
Firstly,
$$
\|\bg_{\e}\|_{L^{d'}(\O_{\e})}  = \|\tilde \bg_{\e} \|_{L^{d'}(\wto_{\e})} \leq  \|\tilde \bg_{\e} \|_{L^{d'}(\R^{d}\setminus \overline T)} =1, \quad \forall \e\in (0,1).
$$
Moreover, since  $\tilde \bg_{\e} $ is actually independent of $\e$, its weak limit in $L^{d'}(\R^{d}\setminus \overline T)$ is itself:
$$
\bg_{\infty} = \tilde \bg_{\e} = \frac{|\nabla \ww|^{d-2}\nabla \ww}{\|\nabla \ww\|_{L^{d}(\R^{d}\setminus \overline T)}^{\frac{d}{d'}}}.
$$
However, the compatibility condition \eqref{neces-cond-p=d} is not satisfied:
\be\label{neces-cond-p=d-fail}
\int_{\R^d\setminus \overline T} \bg_\infty :\nabla {\bf w} \,\dx = \int_{\R^d\setminus \overline T} \frac{|\nabla \ww|^{d-2}\nabla \ww}{\|\nabla \ww\|_{L^{d}(\R^{d}\setminus \overline T)}^{\frac{d}{d'}}}:\nabla {\bf w} \,\dx = \|\nabla \ww\|_{L^{d}(\R^{d}\setminus \overline T)} \neq 0.
\ee

\medskip

Similarly as in Section \ref{sec:1<p<d'}, we can prove the case $p=d$ by using Theorem \ref{thm-sto3} and a dual argument.  We thus compete the proof of Theorem \ref{thm-sto1-new}.

\begin{remark}
We remark that the argument to prove the case $p=d'$ in this section works also for the case $1<p<d'$.

\end{remark}

\section{Generalized restriction operator}\label{sec:res-bog-proof}

This section is devoted to the proof of Theorem \ref{thm-res}, and Corollary \ref{thm-bog} is proved at the end by using the construction of the restriction operator.

The proof is done mainly by employing the construction of Allaire in Section 2.2 in \cite{ALL-NS1}. Thanks to our results  in Theorem \ref{thm-sto}, the generalization of Allaire's construction from the $L^2$ framework to the $L^p$ framework is rather straightforward. For the completion of the results and for the convenience of the readers, we give a brief proof in the following; in particular, we pointed out the differences in the proof between the $L^2$ case and the $L^p$ case.

\subsection{Proof of Theorem \ref{thm-res}}\label{sec:res-proof}

First of all, we recall the construction of Allaire. Let $D$ and $D_\e$ be the domains introduced in Section \ref{sec:domain}. By the distribution of the holes assumed in \eqref{def-holes1} and \eqref{T-Te}, in each cube $\e C_k$,
\be\label{Bk-e}
T_{\e,k}=\e^\a {\cal O}_{k} (T)+x_k \subset\subset B(x_k,b_1 \e)\subset \subset \e C_k.
\ee

For any $\uu \in W_0^{1,p}(D;\R^3)$ with $3/2<p<3$, we define $R_\e (\uu)$ in the following way:
\ba\label{def-Reu}
&R_\e(\uu) (x) :=\uu (x), &&\mbox{for}\  x \in D\setminus\Big( \bigcup_{k\in K_\e} B(x_k,b_1 \e)\Big),\\
&R_\e(\uu) (x) :=\uu_{\e,k} (x), &&\mbox{for} \  x \in B(x_k,b_1 \e)\setminus \overline T_{\e,k},\  \ k\in K_\e,
\ea
where $\uu_{\e,k}$ solves
\be\label{def-uek}\left\{\begin{aligned}
&-\Delta \uu_{\e,k} + \nabla p_{\e,k}  =-\Delta \uu, \ &&\mbox{in}\ B(x_k,b_1 \e)\setminus \overline T_{\e,k},\\
&\dive \uu_{\e,k}=\dive \uu + \frac{1}{|B(x_k,b_1 \e)\setminus \overline T_{\e,k}|}\int_{\overline T_{\e,k}} \dive \uu\, \dx, \ &&\mbox{in}\ B(x_k,b_1 \e)\setminus \overline T_{\e,k},\\
&\uu_{\e,k}= \uu, \ &&\mbox{on}\ \d B(x_k,b_1 \e),\\
&\uu_{\e,k}=0,\ &&\mbox{on}\ \d  T_{\e,k}.
\end{aligned}\right.
\ee

The key point is to prove the following result:
\begin{proposition}\label{prop:uek}
Let $\uu \in W_0^{1,p}(D;\R^3)$ with $3/2<p<3$. The Dirichlet problem \eqref{def-uek} admits  a unique solution $(\uu_{\e,k},p_{\e,k}) \in W^{1,p}(B(x_k,b_1 \e)\setminus \overline T_{\e,k};\R^{3})\times L^{p}_0(B(x_k,b_1 \e)\setminus \overline T_{\e,k})$ satisfying
\ba\label{est-uek-new}
&\|\nabla \uu_{\e,k}\|_{L^p(B(x_k,b_1 \e)\setminus \overline T_{\e,k};\R^{{3\times3}})} + \|p_{\e,k}\|_{L^p(B(x_k,b_1 \e)\setminus \overline T_{\e,k})}  \\
&\qquad\leq C\, \Big(\|\nabla \uu\|_{L^p(B(x_k,b_1 \e);\R^{{3\times3}})}+ \e^{\frac{(3-p)\a-3}{p}}\|\uu\|_{L^p(B(x_k,b_1 \e);\R^{3})} \Big).
\ea

\end{proposition}

To prove Proposition \ref{prop:uek}, we introduce the following two lemmas corresponding to the $L^p$ generalizations of Lemma 2.2.3 and Lemma 2.2.4 in \cite{ALL-NS1}.

\begin{lemma}\label{lem:Alem1}
 Let $1<p<3$, $0<\eta<1/2$ and $T^s\subset B_{1}$ be a simply connected domain of class $C^1$. There exists a linear operator $L$ from $W^{1,p}(B_1)$ to $W^{1,p}(B_1\setminus \eta \overline T^s)$ such that for any $u\in W^{1,p}(B_1)$:
 \ba\label{pt-Alem1}
 & L(u)=u \quad \mbox{on} \ \d B_1,\qquad L(u) =0 \quad \mbox{on} \ \d (\eta T^s),\\
 &\|\nabla L(u)\|_{L^p(B_1\setminus \eta \overline T^s;\R^3)}\leq C\,\Big( \|\nabla u\|_{L^p(B_1;\R^3)}+ \eta^{\frac{3}{p}-1} \| u\|_{L^p(B_1)}  \Big),\nn
 \ea
 where the constant $C$ is independent of $\eta$.

\end{lemma}

\begin{lemma}\label{lem:Alem2}
Let $1<p<3$, $0<\eta<1/2$ and $T^s\subset B_{1}$ be a simply connected domain of class $C^1$.  There exists a linear operator $\CalB_\eta$ from $L_0^p(B_1\setminus \eta \overline T^s)$ to $W_0^{1,p}(B_1\setminus \eta \overline T^s;\R^3)$ such that for any $f\in L^p_0(B_1\setminus \eta \overline T^s)$, there holds
 \ba\label{pt-Alem2}
  \dive \CalB_\eta (f) =f \quad \mbox{in} \ B_1\setminus \eta \overline T^s,\quad \|\CalB_\eta (f)\|_{W_0^{1,p}(B_1\setminus \eta \overline T^s;\R^3)}\leq C\, \|f\|_{L^p(B_1\setminus \eta \overline T^s)},\nn
 \ea
 where the constant $C$ is independent of $\eta$.
\end{lemma}

The proof of above two lemmas can be done by generalizing the proof of Lemma 2.2.3 and Lemma 2.2.4 in \cite{ALL-NS1} from the $L^2$ framework to the $L^p$ framework, so we omit the details. In particular, Lemma \ref{lem:Alem2} can also be proved by observing that $B_1\setminus \eta \overline T^s, \ \eta >0$ are uniform John domains  as $\eta\to 0$ (see \cite{ADM} and \cite{DRS}).

Now we apply Lemma \ref{lem:Alem1} and Lemma \ref{lem:Alem2}, together with the results we obtained in Theorem \ref{thm-sto},   to prove the following lemma:
\begin{lemma}\label{lem:Alem3}
Let $3/2<p<3$, $0<\eta<1/2$ and $T^s\subset B_{1}$ be a simply connected domain of class $C^{2,\b}$. There exists a unique solution $(\vv_\eta,p_\eta) \in W^{1,p}(B_1\setminus \eta\overline T^s;\R^3)\times L^{p}_0(B_1\setminus \eta\overline T^s)$ to the following Dirichlet problem for Stokes equations:
\be\label{def-veta}\left\{\begin{aligned}
&-\Delta \vv_{\eta} + \nabla p_{\eta}  =-\Delta \vv, \ &&\mbox{in}\ B_1\setminus \eta\overline T^s,\\
&\dive \vv_{\eta}=\dive \vv + \frac{1}{|B_1\setminus \eta\overline T^s|}\int_{\eta \overline T^{s}} \dive \vv\, \dx, \ &&\mbox{in}\ B_1\setminus \eta\overline T^s,\\
&\vv_{\eta}= \vv, \ &&\mbox{on}\ \d B_1,\\
&\vv_{\eta}=0,\ &&\mbox{on}\ \d  (\eta T^s),
\end{aligned}\right.
\ee
and there holds the estimate:
 \be\label{est-veta}
\|\nabla \vv_\eta \|_{L^p(B_1\setminus \eta \overline T^s;\R^{3\times3})}+ \|p_\eta \|_{L^p(B_1\setminus \eta \overline T^s)}\leq C\,\big( \|\nabla \vv\|_{L^p(B_1;\R^{3\times3})}+ \eta^{\frac{3}{p}-1}\| \vv \|_{L^p(B_1;\R^3)}  \big),
 \ee
 where the constant $C$ is independent of $\eta$.
\end{lemma}

\begin{proof}[Proof of {\rm Lemma \ref{lem:Alem3}}]
First of all, we observe that the compatibility condition holds:
\ba\label{com-sto-veta}
&\int_{B_1\setminus \eta\overline T^s} \dive \vv_{\eta} \,\dx = \int_{B_1\setminus \eta\overline T^s} \left(\dive \vv + \frac{1}{|B_1\setminus \eta\overline T^s|}\int_{\eta \overline  T^{s}} \dive \vv\, \dx\right) \,\dx\\
&\quad = \int_{B_1} \dive \vv  \,\dx  = \int_{\d B_1} \vv\cdot {\mathbf n}  \,{\rm d}S  = \int_{\d B_1 \cup \d (\eta T^{s})} \vv_{\eta}\cdot {\mathbf n}  \,{\rm d}S.\nn
\ea
This actually indicates, for any fixed $0<\eta<1/2$, there exists a unique solution $(\vv_\eta,p_\eta)$ in $W^{1,p}(B_1\setminus \eta\overline T^s;\R^3)\times L^{p}_0(B_1\setminus \eta\overline T^s)$ to Dirichlet problem \eqref{def-veta} by employing the classical theory for Stokes equations (see Theorem 5.1 in \cite{DM}, Section 3 in \cite{MM} or Theorem 2.9 in \cite{BS}). The key point is to show the quantitative estimate in \eqref{est-veta}, particularly the dependency on $\eta$ as $\eta\to 0$. 

\medskip

By Lemma \ref{lem:Alem1}, the new unknown $\vv_{1,\eta}:= \vv_\eta -L(\vv) $ solves
\be\label{def-veta1}\left\{\begin{aligned}
&-\Delta \vv_{1,\eta} + \nabla p_{\eta}  =-\Delta \vv + \Delta L(\vv),  &&\mbox{in}\ B_1\setminus \eta\overline T^s,\\
&\dive \vv_{1,\eta}= -\dive L(\vv)+ \dive \vv + \frac{1}{|B_1\setminus \eta \overline T^s|}\int_{\eta \overline T^{s}} \dive \vv\, \dx, &&\mbox{in}\ B_1\setminus \eta\overline T^s,\\
&\vv_{1,\eta}= 0, &&\mbox{on}\ \d B_1\cup \d  (\eta T^s).
\end{aligned}\right.\nn
\ee

Define
 $$f:=-\dive L(\vv) + \dive \vv + \frac{1}{|B_1\setminus \eta\overline T^s|}\int_{\eta \overline T^{s}} \dive \vv\, \dx.$$
By virtue of Lemma \ref{lem:Alem1}, direct calculation gives
 \ba\label{est-f-bog} \|f\|_{L^p(B_1 \setminus \eta\overline T^s )}\leq C\,\Big( \|\nabla \vv\|_{L^p(B_1;\R^{3\times3})}+ \eta^{\frac{3}{p}-1} \| \vv\|_{L^p(B_1;\R^3)}  \Big)\nn
\ea
and
\ba\label{int-f-0}
\int_{B_1 \setminus \eta\overline T^s} f\,\dx &= \int_{B_1 \setminus \eta\overline T^s}\left( -\dive L(\vv) + \dive \vv + \frac{1}{|B_1\setminus \eta\overline T^s|}\int_{\eta \overline T^{s}} \dive \vv\, \dx\right)\,\dx\\
&= \int_{B_1 \setminus \eta\overline T^s} -\dive L(\vv)\,\dx + \int_{B_1} \dive \vv \,\dx\\
&= -\int_{\d B_1 \cup \d (\eta T^{s})} L(\vv) \cdot {\mathbf n}  \,{\rm d}S  + \int_{\d B_1} \vv \cdot {\mathbf n}  \,{\rm d}S\\
&= -\int_{\d B_1 } \vv \cdot {\mathbf n}  \,{\rm d}S  + \int_{\d B_1} \vv \cdot {\mathbf n}  \,{\rm d}S\\
&=0.\nn
\ea
Then $f\in L^p_0(B_1\setminus \eta \overline T^s)$ and $\CalB_\eta(f)$ is well-defined, where $\CalB_\eta$ is the linear operator defined in Lemma \ref{lem:Alem2}. Then, by Lemma \ref{lem:Alem2}, the new unknown
$$
\vv_{2,\eta}: = \vv_{1,\eta}-\CalB_\eta(f)= \vv_\eta -L(\vv)-\CalB_\eta(f)
$$
solves
\be\label{def-veta2}\left\{\begin{aligned}
&-\Delta \vv_{2,\eta} + \nabla p_{\eta}  =\dive (-\nabla \vv + \nabla L(\vv)+\nabla \CalB_\eta(f)),  &&\mbox{in}\ B_1\setminus \eta\overline T^s,\\
&\dive \vv_{2,\eta}= 0, &&\mbox{in}\ B_1\setminus \eta\overline T^s,\\
&\vv_{2,\eta}= 0, &&\mbox{on}\ \d B_1\cup \d  (\eta T^s).
\end{aligned}\right.
\ee
We see that \eqref{def-veta2} is a Dirichlet problem for the Stokes equations with the divergence free condition, with the no-slip boundary condition and with a divergence form source term. This falls in the framework of \eqref{2}. We apply Theorem \ref{thm-sto} to conclude that there exists a unique solution $(\vv_{2,\eta},p_\eta) \in  W_0^{1,p}(B_1\setminus \eta\overline T^s;\R^3)\times L_0^{p}(B_1\setminus \eta\overline T^s)$ to \eqref{def-veta2} such that
 \ba\label{est-veta2}
\|\nabla \vv_{2,\eta} \|_{L^p(B_1\setminus \eta \overline T^s;\R^{3\times3})}+\|p_{\eta} \|_{L^p(B_1\setminus \eta \overline T^s)}  & \leq C\, \|-\nabla \vv + \nabla L(\vv)+\nabla \CalB_\eta(f)\|_{L^p(B_1;\R^{3\times3})}\\
&\leq C\,\left( \|\nabla \vv\|_{L^p(B_1;\R^{3\times3})}+ \eta^{\frac{3}{p}-1}\| \vv \|_{L^p(B_1;\R^3)}  \right).\nn
 \ea

Back to $\vv_\eta$, we obtain the estimate \eqref{est-veta} and complete the proof.

\end{proof}

\medskip

Now we are ready to prove Proposition \ref{prop:uek}.
\begin{proof}[Proof of {\rm Proposition \ref{prop:uek}}]
Let $\uu \in W_0^{1,p}(D)$ with $3/2<p<3$ be as in Proposition \ref{prop:uek}.  First of all, we observe that the Dirchlet problem \eqref{def-uek} is compatible:
\ba\label{com-sto-uek}
&\int_{B(x_k,b_1 \e)\setminus \overline T_{\e,k}} \dive \uu_{\e,k} \,\dx= \int_{B(x_k,b_1 \e)\setminus \overline T_{\e,k}} \left(\dive \uu + \frac{1}{|B(x_k,b_1 \e)\setminus \overline T_{\e,k}|}\int_{\overline T_{\e,k}} \dive \uu\, \dx\right) \,\dx\\
&\quad = \int_{B(x_k,b_1 \e)} \dive \uu  \,\dx = \int_{\d B(x_k,b_1 \e)} \uu\cdot {\mathbf n}  \,{\rm d}S = \int_{\d B(x_k,b_1 \e)\cup \d T_{\e,k}} \uu_{\e,k}\cdot {\mathbf n}  \,{\rm d}S.\nn
\ea
Thus, for any fixed $\e>0$, the classical theory implies the existence and uniqueness of solution $(\uu_{\e,k},p_{\e,k})$ in $W^{1,p}(B(x_k,b_1 \e)\setminus \overline T_{\e,k};\R^{3})\times L^{p}_0(B(x_k,b_1 \e)\setminus \overline T_{\e,k})$ to \eqref{def-uek}. 
We consider the following change of unknowns:
\be\label{uek-v-eta}
\vv_{\e,k}(\cdot ) = \vu_{\e,k}(b_1\e \cdot + x_k), \quad q_{\e,k}(\cdot ) = (b_1 \e) p_{\e,k}(b_1\e \cdot + x_k).\nn
\ee
Then
\be\label{def-vek}
(\vv_{\e,k},q_{\e,k}) \in W^{1,p}(B_1\setminus \eta\overline T^s;\R^3)\times L^{p}_0(B_1\setminus \eta\overline T^s),\quad \eta := \e^{\a-1},\ T^s := b_1^{-1}\mathcal{O}_k(T),
\ee
and there holds
\be\label{def-veta-vek}\left\{\begin{aligned}
&-\Delta \vv_{\e,k} +  \nabla q_{\e,k}  =-\Delta \vv, \ &&\mbox{in}\ B_1\setminus \eta\overline T^s,\\
&\dive \vv_{\e,k}=\dive \vv + \frac{1}{|B_1\setminus \eta\overline T^s|}\int_{\eta \overline T^{s}} \dive \vv\, \dx, \ &&\mbox{in}\ B_1\setminus \eta\overline T^s,\\
&\vv_{\e,k}= \vv, \ &&\mbox{on}\ \d B_1,\\
&\vv_{\e,k}=0,\ &&\mbox{on}\ \d  (\eta T^s),
\end{aligned}\right.\nn
\ee
where $\vv$ is defined by
\be\label{def-vv}
\vv(\cdot ) = \vu(b_1\e \cdot + x_k)\in W^{1,p}(B_1;\R^{3}).
\ee

By \eqref{Bk-e} and \eqref{def-vek}, we have $ T^s \subset \overline T^s \subset B_1$. If $\a=1$ such that $\eta=\e^{\a-1}=1$, the domain $B_1\setminus \eta\overline T^s=B_1\setminus \overline T^s$ is a fixed, bounded domain of class $C^1$. We employ the classical theory for Stokes equations in bounded domain (see for instance \cite{DM,MM,BS}) to deduce that there exists a unique solution $(\vv_{\e,k},q_{\e,k}) \in W^{1,p}(B_1\setminus \overline T^s;\R^3)\times L^{p}_0(B_1\setminus \overline T^s)$ satisfying
 \be\label{est-veta3}
\|\nabla \vv_{\e,k}\|_{L^p(B_1\setminus  \overline T^s;\R^{3\times3})}+ \|q_{\e,k} \|_{L^p(B_1\setminus \overline T^s)}\leq C\,\|\vv\|_{W^{1,p}(B_1;\R^{3})}.
 \ee
Due to \eqref{def-vek} and \eqref{def-vv}, the estimate in \eqref{est-veta3} is equivalent to
\ba\label{est-uek-1}
&\|\nabla \uu_{\e,k}\|_{L^p(B(x_k,b_1 \e)\setminus \overline T_{\e,k};\R^{{3\times3}})} + \|p_{\e,k}\|_{L^p(B(x_k,b_1 \e)\setminus \overline T_{\e,k})}  \\
&\qquad\leq C\, \Big(\|\nabla \uu\|_{L^p(B(x_k,b_1 \e);\R^{{3\times3}})}+ \frac{1}{b_1 \e}\|\uu\|_{L^p(B(x_k,b_1 \e);\R^{3})} \Big).
\ea

If $\a>1$, without loss of generality we may assume $\eta:= \e^{\a-1}<1/2$. Then by Lemma \ref{lem:Alem3}, there exists a unique solution $(\vv_{\e,k},q_{\e,k}) \in W^{1,p}(B_1\setminus \eta\overline T^s;\R^3)\times L^{p}_0(B_1\setminus \eta\overline T^s)$ satisfying
 \be\label{est-veta4}
\|\nabla \vv_{\e,k}\|_{L^p(B_1\setminus  \eta\overline T^s;\R^{3\times3})}+ \|q_{\e,k} \|_{L^p(B_1\setminus \eta\overline T^s)}\leq C\,\big( \|\nabla \vv\|_{L^p(B_1;\R^{3\times3})}+ \eta^{\frac{3}{p}-1}\| \vv \|_{L^p(B_1;\R^3)}  \big).
 \ee
Again by \eqref{def-vek} and \eqref{def-vv}, we deduce from \eqref{est-veta4} that
\ba\label{est-uek-2}
&\|\nabla \uu_{\e,k}\|_{L^p(B(x_k,b_1 \e)\setminus \overline T_{\e,k};\R^{{3\times3}})} + \|p_{\e,k}\|_{L^p(B(x_k,b_1 \e)\setminus \overline T_{\e,k})}  \\
&\qquad\leq C\, \Big(\|\nabla \uu\|_{L^p(B(x_k,b_1 \e);\R^{{3\times3}})}+ \frac{1}{b_1 \e} \eta^{\frac{3}{p}-1}\|\uu\|_{L^p(B(x_k,b_1 \e);\R^{3})} \Big).
\ea

By observing
\be\label{def-eta-e}
\e^{-1} \eta^{\frac{3}{p}-1} = \e^{-1}\e^{(\a-1)(\frac{3}{p}-1)} = \e^{\frac{(3-p)\a-3}{p}},
\ee
we obtain our desired result \eqref{est-uek-new} by using \eqref{est-uek-1} and \eqref{est-uek-2}. The proof is completed.

\end{proof}

At the end,  we verify the linear functional $R_\e (\cdot)$ defined by \eqref{def-Reu} and \eqref{def-uek} fulfills the properties stated in Theorem \ref{thm-res}.

Given $\vu \in W_{0}^{1,p}(D;\R^3)$, by Proposition \ref{prop:uek}, we have $R_\e(\vu) \in W_{0}^{1,p}(D_\e;\R^3)$ for which the zero trace property on the boundaries of holes is guaranteed by the construction \eqref{def-uek}.

Given $\vu \in W_{0}^{1,p}(D_\e;\R^3)$, let $\ww:= R_\e(\tilde\vu)$ where $\tilde \vu \in W_{0}^{1,p}(D;\R^3)$ is the zero extension of $\vu$ in $D$ defined by $\eqref{pt-res}_1$. Then the equation \eqref{def-uek} in $u_{\e,k} = R_\e(\tilde\vu)$ in $B(x_k,b_1 \e)\setminus \overline T_{\e,k}$ becomes
\be\label{def-uek-w}\left\{\begin{aligned}
&-\Delta \uu_{\e,k} + \nabla p_{\e,k}  =-\Delta \tilde \uu, \ &&\mbox{in}\ B(x_k,b_1 \e)\setminus \overline T_{\e,k},\\
&\dive \uu_{\e,k}=\dive \tilde \uu,  \ &&\mbox{in}\ B(x_k,b_1 \e)\setminus \overline T_{\e,k},\\
&\uu_{\e,k}= \tilde \uu, \ &&\mbox{on}\ \d B(x_k,b_1 \e)\cup\d  T_{\e,k}.
\end{aligned}\right.
\ee
where we used the property $\tilde \vu=0, \  \dive \tilde \vu=0$ on $\overline T_{\e,k}$. Then, the unique solution $(\uu_{\e,k},p_{\e,k}) \in W^{1,p}(B(x_k,b_1 \e)\setminus \overline T_{\e,k};\R^{3})\times L^{p}_0(B(x_k,b_1 \e)\setminus \overline T_{\e,k})$ to \eqref{def-uek-w} is simply
$$
\uu_{\e,k}= \tilde \vu =\uu, \quad p_{\e,k}=0,\quad \mbox{in $B(x_k,b_1 \e)\setminus \overline T_{\e,k}$}.
$$
This means $R_\e(\tilde \uu)$ coincides with $\uu$ near the holes. Since $R_\e(\tilde \uu)$ also coincides with $\uu$ away from the holes, as in \eqref{def-Reu}, we have $R_\e(\tilde \vu)=\vu$ in $D_\e$.

Given $\vu \in W_{0}^{1,p}(D;\R^3)$ such that $\dive \vu=0$ in $D$, it is straightforward to deduce $\dive R_\e(\vu)=0$ in $D_\e$, due to \eqref{def-Reu} and \eqref{def-uek}.

To complete the proof of Theorem \ref{thm-res}, it is left to verify the estimate in $\eqref{pt-res}_3$. Given $ \uu\in  W_{0}^{1,p}(D;\R^3)$, we calculate
\ba\label{est-Ru1}
&\|\nabla R_\e(\vu)\|^p_{L^p(D_\e;\R^{3\times3})} = \int_{D_\e} |\nabla R_\e(\uu)|^p\,\dx \\
&\quad = \int_{D_\e\setminus \left( \bigcup_{k\in K_\e} B(x_k,b_1 \e)\right)} |\nabla R_\e(\uu)|^p\,\dx + \sum_{k\in K_\e} \int_{B(x_k,b_1 \e)\setminus \overline T_{\e,k}} |\nabla R_\e(\uu)|^p\,\dx\\
&=\int_{D_\e\setminus \left( \bigcup_{k\in K_\e} B(x_k,b_1 \e)\right)} |\nabla \uu|^p\,\dx + \sum_{k\in K_\e} \|\nabla \uu_{\e,k}\|^p_{L^p(B(x_k,b_1 \e)\setminus \overline T_{\e,k};\R^{{3\times3}})},
\ea
where we used the fact that the balls $ B(x_k,b_1 \e), \ {k\in K_\e},$ are pairwise disjoint.  By Proposition \ref{prop:uek} and the inequality $\frac{a^\th+b^\th}{2}\leq (a+b)^\th \leq 2^\th (a^\th+b^\th) \  \mbox{for any $\th>0, \ a>0, \ b>0$},$ we obtain
\ba\label{est-uek-3}
\|\nabla \uu_{\e,k}\|^p_{L^p(B(x_k,b_1 \e)\setminus \overline T_{\e,k};\R^{{3\times3}})}  \leq C\, \Big(\|\nabla \uu\|^p_{L^p(B(x_k,b_1 \e);\R^{{3\times3}})}+ \e^{(3-p)\a-3}\|\uu\|^p_{L^p(B(x_k,b_1 \e);\R^{3})} \Big).
\ea
Applying \eqref{est-uek-3} into \eqref{est-Ru1}, we obtain
\ba\label{est-Ru2}
&\|\nabla R_\e(\vu)\|^p_{L^p(D_\e;\R^{3\times3})} \leq C\, \Big(\|\nabla \uu\|^p_{L^p(D;\R^{{3\times3}})}+ \e^{(3-p)\a-3}\|\uu\|^p_{L^p(D;\R^{{3}})}\Big),\nn
\ea
which implies the desired estimate in $\eqref{pt-res}_3$. We complete the proof of Theorem \ref{thm-res}.

\subsection{Proof of Corollary \ref{thm-bog}}\label{sec:bog-proof}

We prove the existence of such an operator $\CalB_\e$ satisfying the properties stated in Corollary \ref{thm-bog}. 
We will show the operator defined by
\be\label{def-be}
\CalB_\e:= R_\e\circ \CalB_D \circ E,
\ee
fulfills the properties stated in Corollary \ref{thm-bog}. Here $E$ is the zero extension operator from function spaces defined in $D_\e$ to function spaces defined in $D$, $\CalB_D$ is the classical Bogovskii's operator in domain $D$, and $R_\e$ is the restriction operator constructed in the proof of Theorem \ref{thm-res} in Section \ref{sec:res-proof}, precisely in \eqref{def-Reu} and \eqref{def-uek}. This construction can be understood in such a way: extension  + classical Bogovskii operator in homogenized domain + restriction.

Let $3/2<p<3$ and $f\in L_0^p(D_\e)$. It is direct to find that $E$ is a linear operator from $L_0^p(D_\e)$ to $L_0^p(D)$ with operator norm $1$. The classical Bogovskii's operator $\CalB_D$ is a linear operator from $L_0^p(D)$ to $W_0^{1,p}(D;\R^3)$ with operator norm only depending on $p$ and the Lipschitz character of $D$, therefore interdependent of $\e$. Finally by the properties of the restriction operator given in Theorem \ref{thm-res}, we conclude that the operator $\CalB_\e$ defined in \eqref{def-be} is a linear operator from $L_0^p(D_\e)$ to $W_0^{1,p}(D_\e;\R^3)$ such that
\be\label{pt-bog1}
\|\CalB_\e(f)\|_{W_0^{1,p}(D_\e;\R^3)}\leq C\, \left(1+\e^{\frac{(3-p)\a-3}{p}}\right)\|f\|_{L^p(D_\e)}.\nn
\ee

It is left to check the property in $\eqref{pt-bog}_1$. Let $\uu:=\CalB_D(E(f))$, then $\uu \in W_0^{1,p}(D;\R^3)$ satisfies
$$
\dive \uu = f \quad \mbox{in}\ D_\e,\quad \dive \uu = 0 \quad \mbox{on}\  \bigcup_{k\in K_\e} \overline T_{\e,k}.
$$
By the properties of the restriction operator in \eqref{def-Reu} and \eqref{def-uek}, we finally obtain
$$
\dive R_\e(\uu) = \dive \uu = f \quad \mbox{in}\ D_\e.
$$
The proof is completed.






\end{document}